\documentclass[11pt]{amsart}
\usepackage{amsmath}
\usepackage{amssymb}
\usepackage{array}
\usepackage{amscd}
\usepackage{pb-diagram}
\usepackage{comment}
\usepackage{amssymb}
\usepackage{graphicx}
\usepackage{subcaption}
\usepackage{xcolor}

\usepackage[shortlabels]{enumitem}
\usepackage{bm}
\usepackage{hyperref}
\usepackage{tikz-cd}

\usepackage[english]{babel}
\usepackage[utf8]{inputenc}

\def\refeq#1{\if\workingver y(\ref{#1})-[[#1]]\else(\ref{#1})\fi}
\def\refth#1{\if\workingver y\ref{#1}-[[#1]]\else\ref{#1}\fi}
\def\mylabel#1{\if\workingver y\label{#1}{\bf\ \ [[#1]]\ \ }\else\label{#1}\fi}
\def\mybibitem#1{\if\workingver y\bibitem{#1}{\bf\ \ [[#1]]\ \
}\else\bibitem{#1}\fi}

\newfont{\msam}{msam10}
\newfont{\msbm}{msbm10}

\def\articletheorems{
\newtheorem{thm}{Theorem}[section]
\newtheorem{lem}[thm]{Lemma}

\newtheorem{defn}[thm]{Definition}
\newtheorem{cor}[thm]{Corollary}
\newtheorem{prop}[thm]{Proposition}

\newtheorem{algo}{Algorithm}[section] 
\newtheorem{alg}[thm]{Algorithm}  

}

\def\cA{\text{$\mathcal A$}}

\def\cK{\text{$\mathcal K$}}

\def\cM{\text{$\mathcal M$}}

\def\cP{\text{$\mathcal P$}}

\def\cT{\text{$\mathcal T$}}

\def\cV{\text{$\mathcal V$}}

\newcommand{\cl}{\operatorname{cl}}
\newcommand{\opn}{\operatorname{opn}}

\newcommand{\Int}{\operatorname{int}}
\newcommand{\inte}{\operatorname{int}}

\newcommand{\dom}{\operatorname{dom}}

\newcommand{\im}{\operatorname{im}}

\renewcommand{\emptyset}{\varnothing}

\newcommand{\Inv}{\operatorname{Inv}}

\def\begeq#1{\begin{equation}\mylabel{#1}}
\def\endeq{\end{equation}}

\def\mathobj#1{\mbox{$#1$}}

\def\NN{\mathobj{\mathbb{N}}}
\def\PP{\mathobj{\mathbb{P}}}

\def\RR{\mathobj{\mathbb{R}}}

\def\ZZ{\mathobj{\mathbb{Z}}}

\def\rep#1{\mbox{$\langle#1\rangle$}}

\def\setof#1{\mbox{$\{\,#1\,\}$}}

\def\0#1{\hbox{\kern25pt}$ #1 $\\}
\def\1#1{\hbox{\kern40pt}$ #1 $\\}
\def\2#1{\hbox{\kern55pt}$ #1 $\\}
\def\3#1{\hbox{\kern70pt}$ #1 $\\}

\newcounter{li}

\def\begalg#1{\begin{algo}\mylabel{#1}\normalshape:\small\baselineskip 10pt\\}
\def\endalg{\end{algo}}

\def\Figures(include=#1,cat=#2){
  \renewcommand{\textfraction}{.20}
  \renewcommand{\topfraction}{.80}
  \renewcommand{\bottomfraction}{.80}
  \renewcommand{\floatpagefraction}{.80}
  \newcount\figcount
  \figcount=0
  \let\includefigures=#1
  \def\figcat{#2}
}

\def\FigureFromFile[#1][#2](#3)#4
{
  \begin{figure}[htbp]
     \global\advance\figcount by 1
     \if\includefigures y\special{anisoscale #1.wmf, \the\hsize #2}\fi
     \vspace{#2}
     \caption{#4}
     \mylabel{#3}
   \end{figure}
}

\def\FigureFromFileTwoD[#1][#2,#3](#4)#5
{
  \begin{figure}[htbp]
     \global\advance\figcount by 1
     \if\includefigures y\special{anisoscale #1.wmf, #2 #3}\fi
     \vspace{#2}
     \caption{#5}
     \mylabel{#4}
   \end{figure}
}

\def\FigureF<#1>[#2](#3)#4
{
  \begin{figure}[htbp]
     \global\advance\figcount by 1
     \if\includefigures y\special{anisoscale \figcat/fig\number\figcount.wmf,
       \the\hsize #2}
     \fi
     \if\includefigures p
       \leavevmode
       \epsfxsize=\hsize
       \epsffile{#1}
     \fi
     \if\includefigures y
          \vspace{#2}
     \fi
     \caption{#4}
     \mylabel{#3}
   \end{figure}
}

\def\Figure[#1](#2)#3
{
  \begin{figure}[htbp]
     \global\advance\figcount by 1
     \if\includefigures y\special{anisoscale \figcat/fig\number\figcount.wmf,
       \the\hsize #1}
     \fi
     \if\includefigures p
       \leavevmode
       \epsfxsize=\hsize
       \epsffile{fig\number\figcount.eps}
     \fi
     \if\includefigures y
          \vspace{#1}
     \fi
     \caption{#3}
     \mylabel{#2}
   \end{figure}
}

\def\0{\hbox{\kern5pt}}
\def\1{\hbox{\kern20pt}}
\def\2{\hbox{\kern35pt}}
\def\3{\hbox{\kern50pt}}
\def\4{\hbox{\kern65pt}}
\def\5{\hbox{\kern80pt}}
\def\6{\hbox{\kern95pt}}

\let\visiblecomments y 

\newcommand{\currentDate}{\today}
\def\PP{\mathobj{\mathbb{P}}}

\newcommand{\cTop}{\mathcal{T}^{\operatorname{op}{}}}
\newcommand{\Xop}{X^{\operatorname{op}{}}}
\newcommand{\cVop}{\mathcal{V}^{\operatorname{op}{}}}
\newcommand{\op}{\operatorname{op}}
\newcommand{\mouth}{\operatorname{mo}}
\newcommand{\mo}{\operatorname{mo}}
\newcommand{\supp}{\operatorname{supp}}

\def\vclass#1{[#1]_{\cV}}

\newcommand{\lep}[2][]{#2^{\sqsubset#1}}
\renewcommand{\rep}[2][]{#2^{\sqsupset#1}}

\newcommand{\cVp}{\mathcal{V}^+}
\newcommand{\cVm}{\mathcal{V}^-}
\newcommand{\uimp}{\operatorname{uim}^+}
\newcommand{\uimm}{\operatorname{uim}^-}

\newcommand{\inv}{\operatorname{Inv}}
\newcommand{\rank}{\operatorname{rank}}
\newcommand{\con}{\operatorname{Con}}

\newcommand{\mto}{\multimap}
\newcommand{\pto}{\nrightarrow}

\newcommand{\sol}{\operatorname{Sol}}
\newcommand{\esol}{\operatorname{eSol}}
\newcommand{\esolp}{\operatorname{eSol}_\cV^+}
\newcommand{\esolm}{\operatorname{eSol}_\cV^-}
\newcommand{\paths}{\operatorname{Path}}
\newcommand{\pathsv}{\operatorname{Path}_\cV}

\newcommand{\valpha}{\cV^-}
\newcommand{\vomega}{\cV^+}

\newcommand{\Piv}{\Pi_\cV}
\newcommand{\pipl}{\pi_\cV^+}
\newcommand{\pimn}{\pi_\cV^-}

\newcommand{\restr}[1]{|_{#1}}

\graphicspath{ {imgs/} }

\articletheorems

\begin{document}

\author{Micha\l{} Lipi\'nski}
\address{Micha\l{} Lipi\'nski, Division of Computational Mathematics,
  Faculty of Mathematics and Computer Science,
  Jagiellonian University, ul.~St. \L{}ojasiewicza 6, 30-348~Krak\'ow, Poland
}
\email{michal.lipinski@uj.edu.pl}
\author{Jacek Kubica}
\address{Jacek Kubica, Division of Computational Mathematics,
  Faculty of Mathematics and Computer Science,
  Jagiellonian University, ul.~St. \L{}ojasiewicza 6, 30-348~Krak\'ow, Poland
}
\email{jacek.kubica@student.uj.edu.pl}
\author{Marian Mrozek}
\address{Marian Mrozek, Division of Computational Mathematics,
  Faculty of Mathematics and Computer Science,
  Jagiellonian University, ul.~St. \L{}ojasiewicza 6, 30-348~Krak\'ow, Poland
}
\email{Marian.Mrozek@uj.edu.pl}
\author{Thomas Wanner}
\address{Thomas Wanner, Department of Mathematical Sciences,
George Mason University, Fairfax, VA 22030, USA
} \email{twanner@gmu.edu}
\date{today}
\thanks{
Research of  M.L.\ was partially supported by
  the Polish National Science Center under Preludium Grant No. 2018/29/N/ST1/00449.
M.M.\ was partially supported by
  the Polish National Science Center under Ma\-estro Grant No. 2014/14/A/ST1/00453.
   T.W.\ was partially supported by NSF grant 
   DMS-1407087 and by the Simons Foundation under Award 581334.
}
\subjclass[2010]{Primary: 37B30; Secondary: 37E15, 57M99, 57Q05, 57Q15.}
 \keywords{Combinatorial vector field, finite topological space,
 discrete Morse theory,  isolated invariant set, Conley theory.}

\title[Conley-Morse-Forman theory on finite topological spaces]
{Conley-Morse-Forman theory for generalized combinatorial multivector fields on finite topological spaces}

\date{Version compiled on \currentDate}

\begin{abstract}
  We generalize and extend the Conley-Morse-Forman theory for combinatorial multivector fields introduced in \cite{Mr2017}.
  The generalization consists in dropping the restrictive assumption in \cite{Mr2017} that every multivector has a unique maximal element.
  The extension is from the setting of Lefschetz complexes to the more general situation of finite topological spaces.
  We define isolated invariant sets, isolating neighbourhoods, Conley index and Morse decompositions. 
  We also establish the additivity property of the Conley index and the Morse inequalities.
  
\end{abstract}

\maketitle

\begin{center}
Version compiled on \today
\end{center}


\section{Introduction}
\label{sec:intro}

The combinatorial approach to dynamics has its origins in two papers by Robin Forman \cite{Fo98a,Fo98b} published in the late 1990s.
Central to the work of Forman is the concept of a combinatorial vector field.
One can think of a combinatorial vector field as a partition of a collection of cells of a cellular complex into combinatorial vectors which may be singletons (critical vectors or critical cells) or doubletons such that one element of the doubleton is a face of codimension one
of the other (regular vectors).
The original motivation of Forman was the presentation of a combinatorial analogue
of classical Morse theory.
However, soon the potential for applications of such an approach was discovered in data science.
Namely, the concept of combinatorial vector field enables direct applications
of the ideas of topological dynamics to data and eliminates the need of the cumbersome
construction of a classical vector field from data.

Recently, B. Batko, T. Kaczynski, M. Mrozek and Th. Wanner \cite{BKMW2020, KaMrWa2016},
in an attempt to build formal ties between
the classical and combinatorial Morse theory,
extended the combinatorial theory of Forman to Conley theory \cite{Conley1978}, a generalization of Morse theory.
In particular, they defined the concept of an isolated invariant set, the Conley index and Morse decomposition
in the case of a combinatorial vector field on the collection of simplices of a simplicial
complex. Later, M. Mrozek \cite{Mr2017} observed that certain dynamical structures, in particular
homoclinic connections, cannot have an analogue for combinatorial vector fields and as a remedy proposed
an extension of the concept of combinatorial vector field, a combinatorial multivector field.
We recall that in the collection of cells of a cellular complex
there is a natural partial order induced by the face relation.
Every combinatorial vector in the sense of Forman is convex with respect to this partial order.
A combinatorial multivector in the sense of \cite{Mr2017} is defined as a convex collection of cells
with a unique maximal element, and a combinatorial multivector field is then
defined as a partition of cells into multivectors.
The results of \cite{Mr2017} were presented in the algebraic setting of chain complexes
with a distinguished basis (Lefschetz complexes), an abstraction of the chain complex of a cellular complex
already studied by S.\ Lefschetz \cite{Le1942}.
The results of Forman were earlier generalized to the setting of Lefschetz complexes in \cite{JW2009,Ko2005,Sk2006}.

The aim of this paper is a threefold advancement of the results of \cite{Mr2017}.
We generalize the concept of combinatorial multivector field by lifting the assumption
that a multivector has a unique maximal element. This assumption was introduced in
\cite{Mr2017} for technical reasons but turned out to be a barrier for adapting the
techniques of continuation in topological dynamics to the combinatorial setting.
We change the setting from Lefschetz complexes to the more general finite topological spaces.
The combinatorial Morse theory in such a setting was introduced in \cite{Mi2012}.
And, following the ideas of \cite{DJKKLM2017}, we define the dynamics associated with a combinatorial multivector field
in a less restrictive way, better adjusted to persistence theory for combinatorial dynamics.

In this extended and generalized setting we define the concepts of isolated invariant set and Conley index.
We also define attractors, repellers, attractor-repeller pairs and Morse decompositions
and provide a topological characterization of attractors and repellers.
Furthermore, we prove the Morse equation  for Morse decompositions, and finally
deduce from it the Morse inequalities.

The organization of the paper is as follows.
  In Section~\ref{sec:main-results} we present the main results of the paper for an elementary geometric example.
  In Section~\ref{sec:preliminaries} we recall basic concepts and facts needed in the paper.
  Section~\ref{sec:dcmvf} is devoted to the study of the dynamics of combinatorial
    multivector fields and the introduction of isolated invariant sets.
  In Section~\ref{sec:index-pairs-conley-index} we define index pairs and the Conley index. 
  In Section~\ref{sec:attr-rep-lim} we investigate limit sets, attractors and repellers in the combinatorial setting. 
  Finally, Section~\ref{sec:morse-decomposition} is concerned with Morse decompositions and Morse inequalities for combinatorial multivector fields.


\section{Main results}\label{sec:main-results}

In this section we present the main ideas and results of the paper using a simple simplicial example.  
We also indicate the main conceptual differences between our combinatorial approach and the classical theory.

\subsection{A simple combinatorial flow on a simplicial complex}

The results of this paper apply to an arbitrary, finite, $T_0$ topological space. 
Among natural examples of such spaces are collections of cells of a finite cellular complex, in particular a simplicial complex.

Let $X$ be such a collection of simplices of a finite simplicial complex.
The face relation between cells is a partial order in $X$ (see Fig. \ref{fig:complex-example}) which by the Alexandrov Theorem (see Theorem \ref{thm:alexandroff}), induces a $T_0$ topology on $X$.
This topology is closely related to the $T_2$ topology of the polytope of the simplicial complex, that is, the union of the simplices in $X$. 
We can see it by identifying the nodes of this poset with open simplices.
Then, the set $A\subset X$ is open (respectively closed) in the $T_0$ topology of $X$ if and only if the union of the corresponing open simplices is open (respectively closed) in the Hausdorff topology of the polytope of $X$.

\begin{figure}
  \includegraphics[width=0.75\textwidth]{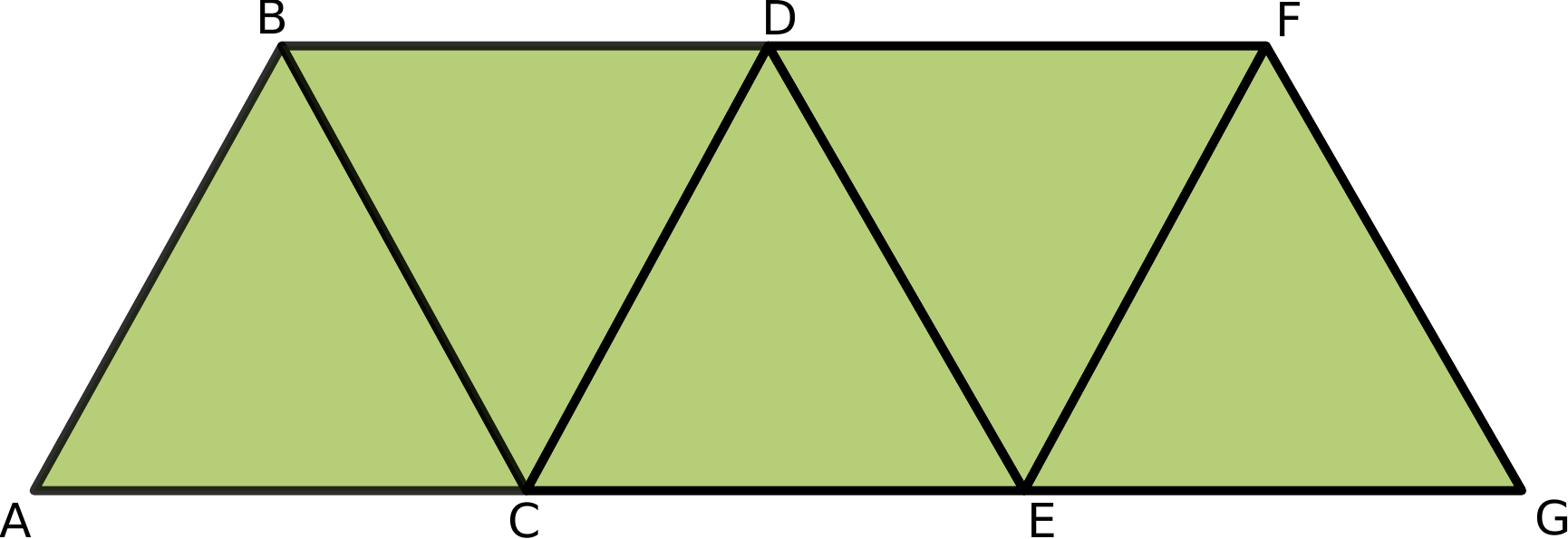}

  \vspace{0.4cm}
  \includegraphics[width=\textwidth]{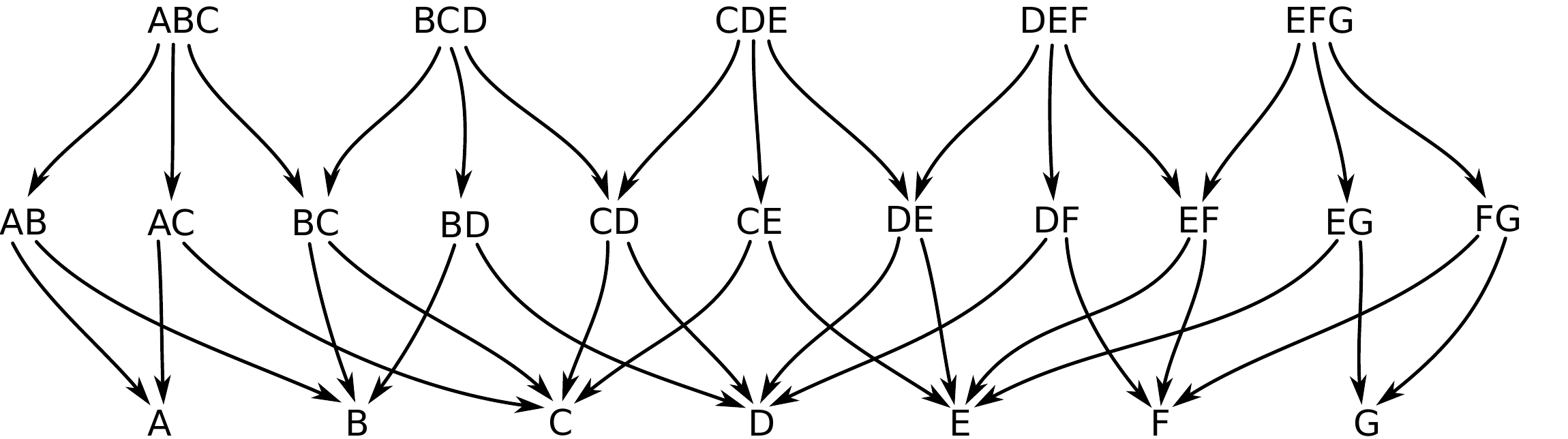}
  \caption{An example of a simplicial complex (top) and the poset (a finite $T_0$ topological space, bottom) induced by its face relation.}
  \label{fig:complex-example}
\end{figure}

Our combinatorial counterpart of the classical concept of dynamical system takes the form of the combinatorial flow induced by a combinatorial multivector field defined as follows.
By a \emph{multivector} we mean a nonempty, locally closed (or convex in terms of posets, see Proposition \ref{prop:lcl-in-ftop}) subset of $X$. 
A \emph{combinatorial multivector field} $\cV$ is a partition of $X$ into multivectors.  
We say that a multivector $V$ is \emph{critical} if $H(\cl V, \cl V\setminus V)\neq 0$,
where $H(\cdot,\cdot)$ denotes relative singular homology. Otherwise we call $V$ \emph{regular}.
One can think of a multivector as a "black box" where the inner dynamics cannot be determined.  
The only available information is the direction of the flow at the boundary of a multivector. 
In particular, our construction assumes that the flow may exit the closure of a multivector $V$ only through $\mo V:=\cl V\setminus V$ and may enter the closure of $V$ through $V$.
Hence, $\cl V$ may be interpreted as an isolating block with exit set $\mo V$ and the relative homology $H(\cl V, \mo V)$ may be interpreted as the Conley index of $\cl V$.
For the definition of Conley index and isolating block in the classical setting
see~\cite{Conley1978, conley:easton:71a, stephens:wanner:14a}.
\begin{figure}[b]
  \includegraphics[width=\textwidth]{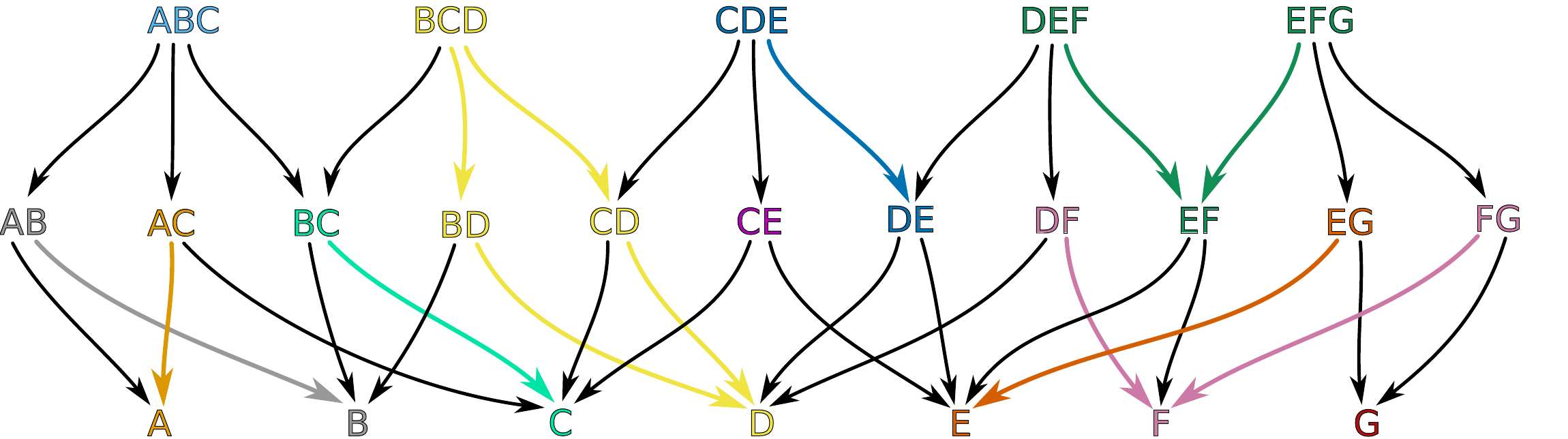}
  \caption{A partition of a poset into multivectors (convex subsets). Nodes as well as corresponding arrows of each multivector are highlighted with a distinct color. }
  \label{fig:mvf-poset-example}
\end{figure}

Figure \ref{fig:mvf-poset-example} shows an example of a multivector field on the poset in Figure~\ref{fig:complex-example}.
It consists of the following multivectors:
\begin{center}
  $\{A, AC\}, \{ABC\}, \{B, AB\}, \{C, BC\}, \{CE\}, \{D, BD, CD, BCD\},$\\
  $\{DE, CDE\}, \{E, EG\}, \{EF, DEF, EFG\}, \{F, DF, FG\}, \{G\}$
\end{center}
Every multivector in Figure \ref{fig:mvf-poset-example} is highlighted with a different color. 
The criticality of a multivector $V$, as a subset of a finite topological space, 
can be easily determined using the order complex of $\cl V$ (see Proposition \ref{prop:simplicial_rel_hom_of_posets}).
 There are five critical multivectors in the example in Figure \ref{fig:mvf-poset-example}:
\[ \{ABC\},\ \{CE\},\ \{DEF,EF,EFG\},\ \{DF,F,FG\},\ \{G\}.\]
Figure \ref{fig:mvf-complex-example} presents the same multivector field visualized in the polytope of the simplicial complex. 
Yellow regions indicate multivectors.  
With respect to the "black box" interpretation of a multivector 
the dotted part of the boundary of a multivector indicates the outward-directed flow while the solid part of the boundary indicates the inward flow.

We define the combinatorial flow associated with the multivector field as the multivalued map
$\Pi_\cV:X\multimap X$ given by
\[ \Pi_\cV(x):= \cl x \cup [x]_\cV\]
where $[x]_\cV$ denotes the unique multivector in $\cV$ containing $x$.  
The first, closure component of the union indicates that, by default, the flow moves towards the boundary of the simplex $x$.
The second component reflects the "black box" nature of a multivector: we cannot exclude that a given point can reach any other point within the same multivector.

\begin{figure}[t]
  \includegraphics[width=\textwidth]{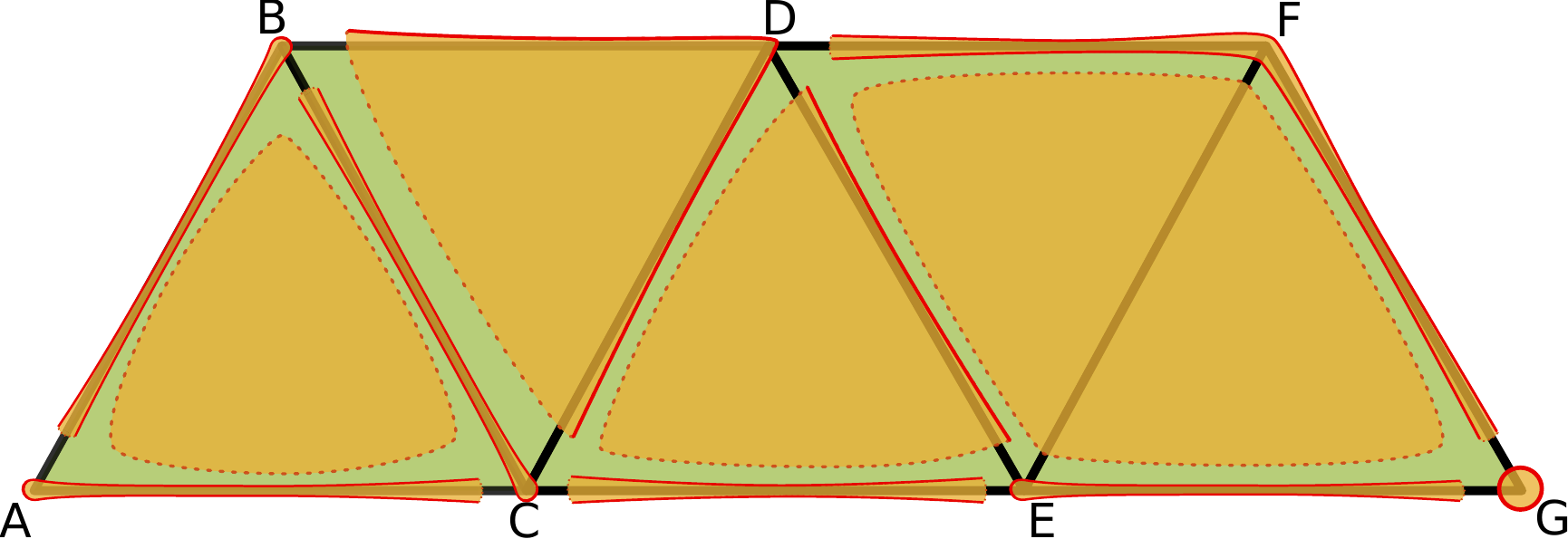}
  \caption{A geometric visualization of the combinatorial multivector field in Figure \ref{fig:mvf-poset-example}. 
  A multivector may be considered as a "black box" whose dynamics is known only via splitting its boundary into the exit and entrance parts.}
  \label{fig:mvf-complex-example}
\end{figure}

\begin{figure}[t]
  \includegraphics[width=\textwidth]{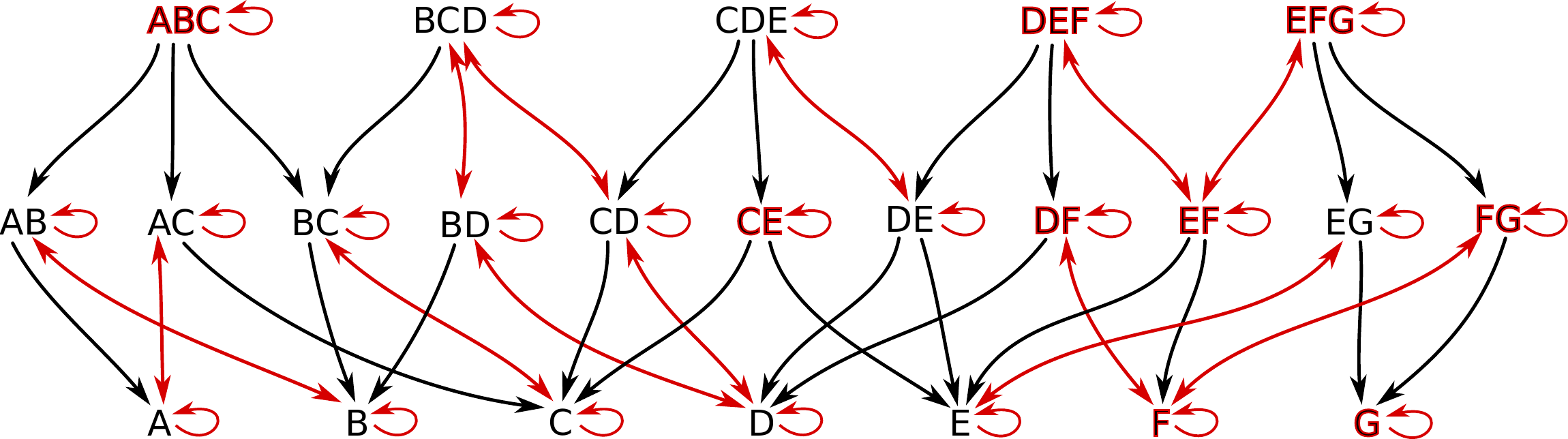}
  \caption{
    The combinatorial flow $\Piv$ of the multivector field in Figures \ref{fig:mvf-poset-example} and \ref{fig:mvf-complex-example} represented as the digraph $G_\cV$.
    Downward arrows are induced by the closure components of $\Piv$.
    Bi-directional edges and self-loops reflect dynamics within multivectors. 
    For clarity, we omit edges that can be obtained by between-level transitivity, e.g., the bi-directional connection between node $D$ and $BCD$. 
    The nodes of critical multivectors are bolded in red.
  }\label{fig:mvf-graph-example}
\end{figure}

The combinatorial flow $\Piv$ may be interpreted as a directed graph $G_\cV$ whose nodes are the simplices in $X$ and there is a directed arrow from $x$ to $y$ whenever $y\in\Piv(x)$. 
The graph $G_\cV$ for our example is presented in Figure \ref{fig:mvf-graph-example}.
The family of paths in the graph $G_\cV$ may be identified with the family of combinatorial solutions of the combinatorial flow $\Piv$, that is, maps $\gamma:A\rightarrow X$ such that $A\subset \ZZ$ is a bounded or unbounded $\ZZ$-interval and $\gamma(k+1)\in\Piv(\gamma(k))$ for all $k,k+1\in A$ (see Section \ref{subs:solutions} for precise definitions). 
We distinguish \emph{essential solutions}. An essential solution is a solution $\gamma$ 
such that if $\gamma(t)$ belongs to a regular multivector $V\in\cV$ then there exist a $k>0$ and an $l<0$ such that $\gamma(t+k),\gamma(t+l)\not\in V$. 
We restrict this assumption to regular multivectors, because for any critical multivector $V$ we have $H(\cl V, \mo V) \neq 0$, which, when interpreted as a nontrivial Conley index, implies that at least one solution stays inside $V$.
We denote by $\esol_\cV(x,A)$ the family of all essential solutions of $\cV$ contained in a set $A\subset X$ and passing through a point $x\in A$. 
An example of an essential solution $\gamma:\ZZ\rightarrow X$ for the multivector field in Figure \ref{fig:mvf-complex-example} is given by:
\[
  \gamma(t) = \begin{cases}
    CE  & t < 0,\\
    E & t \in \{0,2,3\},\\
    EG  & t \in \{1, 4\},\\
    G & t > 4.
    \end{cases} 
\]
We say that a set $S\subset X$ is \emph{invariant} if every $x\in S$ admits an essential solution through $x$ in $S$, that is, if $\esol_\cV(x,S)\neq\emptyset$. 
We say that $S$ is an \emph{isolated invariant} set if there exists a closed set $N$, called an \emph{isolating set} such that $S\subset N$, $\Piv(S)\subset N$ and every path in $N$ with endpoints in $S$ is a path in $S$. 
Note that our concept of isolating set is weaker than the classical concept of isolating neighborhood, because the maximal invariant subset of $N$ may not be contained in the interior of $N$.
The need of a weaker concept is motivated by the tightness in finite topological spaces.
In particular, an isolated invariant set $S$ may intersect the closure of another isolated invariant set $S'$ and be disjoint but not disconnected from $S'$.
For example, the sets $S_1:=\{A, AC, C, BC, B, AB\}$ and $S_2:=\{ABC\}$ are both isolated invariant sets isolated respectively by $N_1:=S_1$ and $N_2:=\cl S_1 = S_1\cup S_2$.
Observe that $S_1\subset N_2$.
Thus, the isolating set in the combinatorial setting of finite topological spaces is a relative concept.
Therefore, one has to specify each time which invariant set is considered as being isolated by a given isolating set. 

Given an isolated invariant set $S$ of a combinatorial multivector field $\cV$ we define index pairs similarly to the classical case (Definition \ref{def:index_pair}), we prove that $(\cl S, \mo S)$ is one of the possibly many index pairs for $S$ (Proposition \ref{prop:minimal_index_pair}) and we show that the homology of an index pair depends only on $S$, but not on the particular index pair (Theorem \ref{thm:index_pairs_isomorphism}).
This enables us to define the Conley index of an isolated invariant set $S$ (Definition \ref{def:iso-inv-set}) and the associated Poincar\'{e} polynomial
\[
  p_S(t):=\sum_{t=0}^\infty \beta_i(S)t^i,
\]
where $\beta_i(S):=\rank H_i(\cl S,\mo S)$ denotes the $i$th Betti number of the Conley index of $S$.
In our example in Figure \ref{fig:mvf-complex-example}, the Poincar\'{e} polynomials of the isolated invariant sets $S_1=\{A,AC,C,BC,B,AB\}$ and $S_2=\{ABC\}$ are respectively $p_{S_1}(t)=1+t$ and $p_{S_2}(t)=t^2$.

As in the classical case we define Morse decompositions (Definition \ref{def:morse_decomposition}).
Unlike the classical case, for a combinatorial multivector field $\cV$ we prove that the strongly connected components of the directed graph $G_\cV$ which admit an essential solution constitute the minimal Morse decomposition of $\cV$ (Theorem \ref{thm:sccs_morse_decomp}). 
For the example in Figure \ref{fig:mvf-complex-example} the minimal Morse decomposition consists of six isolated invariant sets:
\begin{align*}
  M_1 &= \{ A, AC, C, BC, B, AB\},\\
  M_2 &= \{ ABC\},\\
  M_3 &= \{ CE\},\\
  M_4 &= \{ DEF, EF, EFG\},\\
  M_5 &= \{ DF,F,FG\},\\
  M_6 &= \{ G\}.
\end{align*}
We say that an isolated invariant set $S$ is an \emph{attractor} (respectively a \emph{repeller}) if all solutions originating in it stay in $S$ in forward (respectively backward) time. 
There are two attractors in our example: $M_1$ is a periodic attractor, and $M_6$ is an attracting stationary point. 
Sets $M_2$ and $M_4$ are repellers, while $M_3$ and $M_5$ have characteristics of a saddle.

\begin{figure}[b]
  \includegraphics[width=0.4\textwidth]{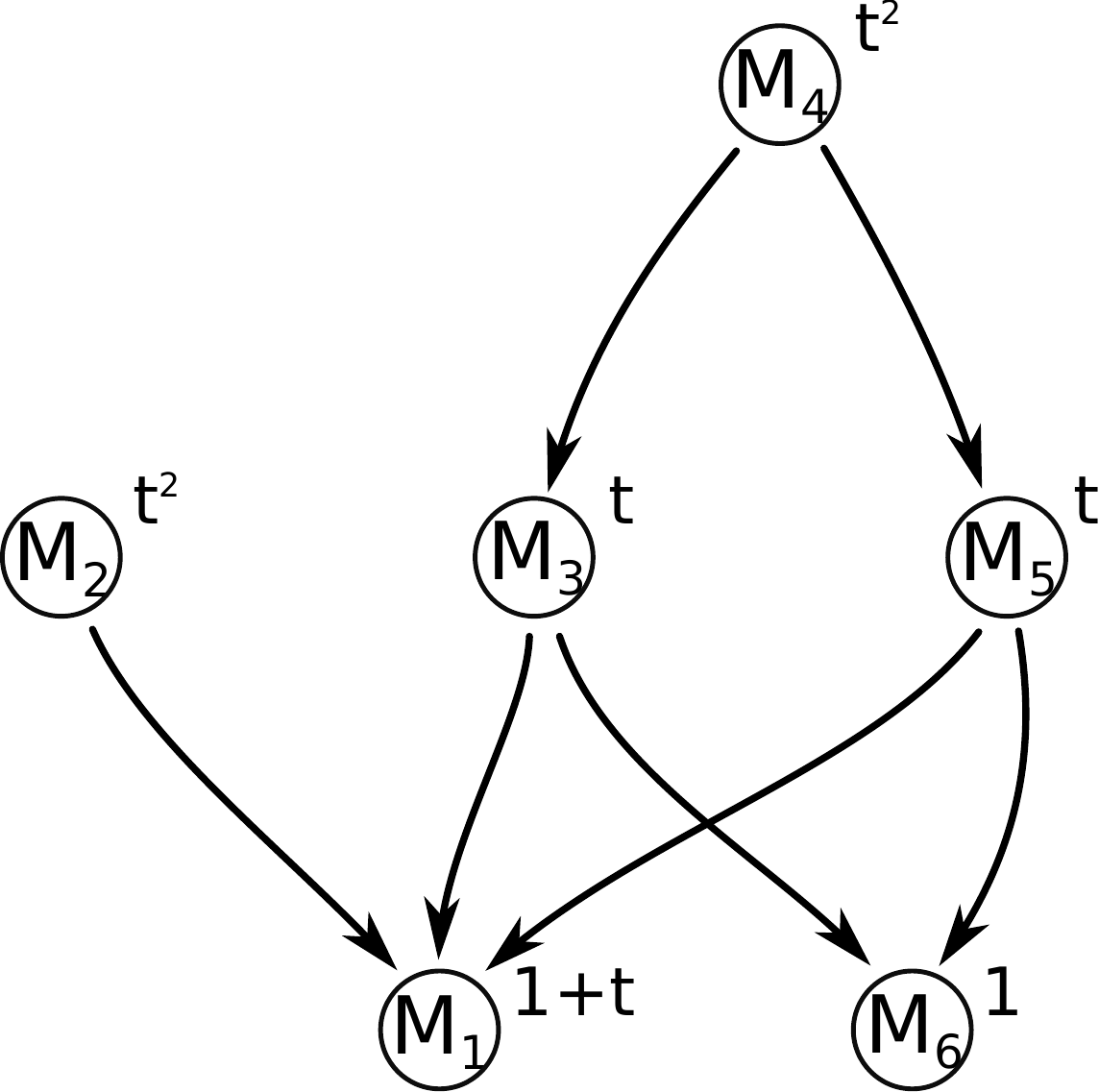}
  \caption{The Conley-Morse graph for the example in Figure \ref{fig:mvf-graph-example}.}
  \label{fig:conley-morse-graph}
\end{figure}

If there exists a path originating in $M_i$ and terminating in $M_j$, we say that there is a \emph{connection} from $M_i$ to $M_j$. 
The connection relation induces a partial order on Morse sets. 
The associated poset with nodes labeled with Poincar\'{e} polynomials is called the \emph{Conley-Morse graph} of the Morse decomposition, see also \cite{arai:etal:09a, bush:etal:12a}.

The Conley-Morse graph of the minimal Morse decomposition of the combinatorial multivector field in Figure \ref{fig:mvf-complex-example} is presented in Figure \ref{fig:conley-morse-graph}. 
The Morse equation (see Theorem \ref{thm:morse-equation}) for this Morse decomposition takes the form:
\[ 2t^2 + 3 t + 2 = 1 + (1+t) (1+2t).\]
As this brief overview of the results of this paper indicates, at least to some extent it is possible to construct a combinatorial analogue of the classical topological dynamics.
Such an analogue may be used to construct algorithmizable models of sampled dynamical systems as well as tools for computer assisted proofs in dynamics.
Translation of the problems in combinatorial dynamics to the language of the directed graph $G_\cV$ facilitates the algorithmic study of the models. 
However, we emphasize that the models cannot be reduced just to graph theory.
What is essential in the presented theory is the fact that the set of vertices of the directed graph constitutes a finite topological space.


\section{Preliminaries}\label{sec:preliminaries}

\subsection{Sets and maps}\label{sec:sets-and-maps}
We denote the sets of integers, non-negative integers, non-positive integers, and positive integers, respectively, by $\ZZ$, $\ZZ^+$, $\ZZ^-$, and~$\NN$. 
Given a set $A$, we write~$\#A$ for the number of elements in $A$ and we denote by $\cP(A)$ the family of all subsets of $X$.
We write $f:X\nrightarrow Y$ for a partial map from~$X$ to~$Y$, that is, a map defined on a subset $\dom{f}\subset X$, called the \emph{domain} of~$f$, and such that the set of values of $f$, denoted $\im f$, is contained in $Y$.

\emph{A multivalued map} $F: X\multimap Y$ is a map $F: X\rightarrow \cP(Y)$ which assigns to every point $x\in X$ a subset $F(x)\subset Y$. Given $A\subset X$, the \emph{image} of $A$ under $F$ is defined by
\begin{equation*}
  F(A):=\bigcup_{x\in A}F(a).
\end{equation*}
By the \emph{preimage} of a set $B\subset Y$ with respect to $F$ we mean the large preimage, that is,
\begin{equation}\label{eq:large_premage}
  F^{-1}(B) := \left\{x\in X\ \mid\ F(x)\cap B\neq\emptyset\right\}.
\end{equation}
In particular, if $B=\{y\}$ is a singleton, we get
\begin{equation*}
  F^{-1}(\{y\}) := \left\{x\in X\ \mid\ y\in F(x)\right\}.
\end{equation*}
Thus, we have a multivalued map $F^{-1}:Y\multimap X$ given by $F^{-1}(y):=F^{-1}(\{y\})$.  
We call it the \emph{inverse} of $F$.

\subsection{Relations and digraphs}
Recall that a binary relation or briefly a relation in a space $X$ is a subset $E\subset X\times X$. We write $xEy$ as a shorthand for $(x,y)\in E$.
The \emph{inverse} of $E$ is the relation 
\begin{equation*}
  E^{-1}:=\left\{ (y,x)\in X\times X\mid xEy\right\}.
\end{equation*}
Given a relation $E$ in $X$, the pair $(X,E)$ may be interpreted as a {\em directed graph} ({\em digraph}) with vertex set $X$, and edge set $E$.

Relation $E$ may also be considered as a multivalued map $E:X\multimap X$ with $E(x):=\{y\in X\mid xEy\}$.  
Thus, the three concepts: binary relation, multivalued map and directed graph are, in principle, the same and in this paper will be used interchangeably.

We recall that a \emph{path} in a directed graph $G=(X,E)$ is a sequence $x_0, x_1,\dots,x_k$ of vertices such that $(x_{i-1}, x_i)\in E$ for $i=1,2,\dots k$.
The path is \emph{closed} if $x_0 = x_k$. 
A closed path consisting of two elements is a \emph{loop}.
Thus, an $x\in X$ is a loop if and only if $x\in E(x)$.  
We note that loops may be present at some vertices of $G$ but at some other vertices they may be absent.

A vertex is \emph{recurrent} if it belongs to a closed path.
In particular, if there is a loop at $x\in X$, then $x$ is recurrent.
The digraph $G$ is {\em recurrent} if all of its vertices are recurrent.
We say that two vertices $x$ and $y$ in a recurrent digraph $G$ are equivalent if there is a path from $x$ to $y$ and a path from $y$ to $x$ in $G$. 
Equivalence of recurrent vertices in a recurrent digraph is easily seen to be an equivalence relation.
The equivalence classes of this relation are called \emph{strongly connected components} of digraph $G$.
They form a partition of the vertex set of $G$.

We say that a recurrent digraph $G$ is \emph{strongly connected} if it has exactly one strongly connected component.
A non-empty subset $A\subset X$ is \emph{strongly connected} if $(A, E\cap A\times A)$ is strongly connected. 
In other words, $A\subset X$ is strongly connected
if and only if for all $x,y\in A$ both~$\pathsv(x,y,A)$ and~$\pathsv(y,x,A)$ are non-empty.

\subsection{Posets}
\label{sec:posets}

Let $X$ be a finite set. We recall that a reflexive and transitive relation $\leq$ on $X$ is a \emph{preorder} and the pair $(X, \leq)$ is a \emph{preordered set}. 
If $\leq$ is also antisymmetric, then it is a \emph{partial order} and $(X,\leq)$ is a \emph{poset}. 
A partial order in which any two elements are comparable is a \emph{linear} (\emph{total}) \emph{order}.

Given a poset $(X,\leq)$, a set $A\subset X$ is \emph{convex} if $x\leq y\leq z$ with $x,z\in A,\ y\in X$ implies $y\in A$. 
It is an \emph{upper set} if $x\leq y$ with $x\in A$ and $y\in X$ implies $y\in A$.
Similarly, $A$ is a \emph{down set} with respect to $\leq$ if $x\leq y$ with $y\in A$ and $x\in X$ implies $x\in A$.
A \emph{chain} is a totally ordered subset of a poset.

The \emph{order complex} of $(X,\leq)$, denoted $\cK(X)$, is the abstract simplicial complex consisting of all nonempty chains of $X$. 
We denote its \emph{geometric realization} or \emph{polytope} by $|\cK(X)|$.
Note that the geometric realization is unique up to a homeomorphism.
Every point $\alpha\in|\cK(X)|$ may be represented as the \emph{barycentric combination} $\alpha=t_1x_1 + t_2x_2 +...+t_nx_n$ where $\sum_{i=1}^{n}t_i=1$, 
  the coefficients $t_i$ are positive and $x_1<x_2<...<x_n$ is a chain in $(X, \leq)$. 
This chain is called the \emph{support} of $\alpha$ and denoted $\supp{\alpha}=\{x_1, x_2,\dots,x_n\}$.
Given preordered sets $(X,\leq)$ and $(Y,\leq)$ a map $f:X\rightarrow Y$ is called \emph{order-preserving} if $x\leq x'$ implies $f(x)\leq f(x')$ for all $x,x'\in X$.

\begin{figure}
    \begin{tabular}{m{0.45\textwidth} m{0.45\textwidth}} 
  \includegraphics[width=0.45\textwidth]{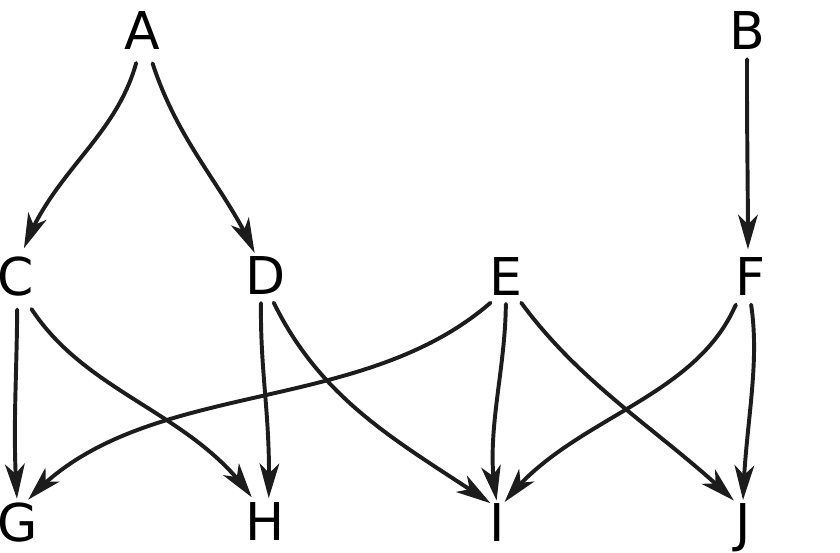} & \includegraphics[width=0.45\textwidth]{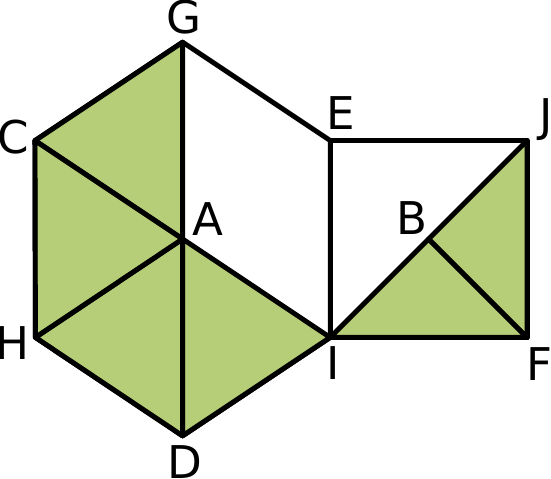} \\
\end{tabular}
  \caption{Left: an example of a poset (a finite topological space). Right: the order complex of the finite topological space from the left panel.}
\end{figure}

\begin{prop}\label{prop:order-complex-oper}
  Let $(X,\leq)$ be a poset and let $A, B\subset X$.Then 
  \begin{equation}\label{eq:order-complex-oper-cap}
    \cK(A\cap B) = \cK(A)\cap\cK(B) \quad \text{and}\quad |\cK(A\cap B)| = |\cK(A)|\cap|\cK(B)|.
  \end{equation}
  Moreover, if $A$ and $B$ are down sets, then
  \begin{equation}\label{eq:order-complex-oper-cup}
    \cK(A\cup B) = \cK(A)\cup\cK(B) \quad \text{and}\quad |\cK(A\cup B)| = |\cK(A)|\cup|\cK(B)|.
  \end{equation}
\end{prop}
\begin{proof}
  Property (\ref{eq:order-complex-oper-cap}) and inclusions $\cK(A)\cup\cK(B)\subset\cK(A\cup B)$, $|\cK(A)|\cup|\cK(B)|\subset|\cK(A\cup B)|$ are straightforward.
  To see that 
  $\cK(A\cup B) \subset \cK(A)\cup\cK(B)$ and $|\cK(A\cup B)| \subset |\cK(A)|\cup|\cK(B)|$
  consider a chain $x_1<x_2<\dots < x_k$ in $\cK(A\cup B)$.
  Without loss of generality we may assume that $x_k\in A$.
  Since $A$ is a down set, we see that $x_1<x_2<\dots < x_k$ is a chain in $\cK(A)\subset\cK(A)\cup\cK(B)$.
  Hence, $\cK(A\cup B)\subset\cK(A)\cup\cK(B)$. The inclusion $|\cK(A\cup B)| \subset |\cK(A)|\cup|\cK(B)|$ follows.
\end{proof}

For $A\subset X$ we write
\begin{align*}
  A^\leq&:=\{a\in X\ |\ \exists_{b\in A}\ a\leq b\},\\
  A^< &:=A^\leq\setminus A.
\end{align*}

\begin{prop}\label{prop:A-leq-down-set}
  Let $(X,\leq)$ be a poset and let $A\subset X$ be a convex set. Then the sets $A^\leq$ and $A^<$ are down sets.
\end{prop}
\begin{proof}
  Clearly, $A^\leq$ is a down set directly from the definition.
  To see it for $A^<$, consider an $a\in A^<$ and a $b\in X$ such that $b< a$. 
  The definitions of $A^\leq$ and $A^<$ imply that there exists a $c\in A$ such that $a\leq c$.
  Since $A^\leq$ is a down set we also have $b\in A^\leq$. 
  We cannot have $b\in A$, because otherwise $b< a< c$ which contradicts the convexity of $A$. Hence, $b\in A^<$ which proves that $A^<$ is a down set.
\end{proof}

\subsection{Finite topological spaces}
\label{sec:ftop}

Given a topology $\cT$ on $X$, we call $(X,\cT)$ a topological space.
When the topology $\cT$ is clear from the context we also  refer to $X$ as a topological space.
We denote the {\em interior} of $A\subset X$ with respect to $\cT$ by  $\inte_\cT A$ and the {\em closure} of $A$ with respect to $\cT$ by  $\cl_\cT A$.
We define the {\em mouth} of $A$ as the set $\mo_\cT A := \cl_\cT A\setminus A$.  
We say that $X$ is a {\em finite topological space} if $X$ is a finite set.  
If $X$ is finite, we also distinguish the \emph{minimal open superset} (or \emph{open hull}) of $A$ as the intersection of all the open sets containing $A$.
We denote it by $\opn_\cT A$. 
We note that when $X$ is finite then the family $\cTop:=\{X\setminus U\mid U\in\cT\}$ of closed sets is also a topology on $X$, called \emph{dual} or \emph{opposite topology}.
The following Proposition is straightforward.
\begin{prop}\label{prop:dual-top}
  If $(X,\cT)$ is a finite topological space then for every set $A\subset X$ we have $\opn_\cT A=\cl_{\cTop} A$.
\end{prop}

If $A=\{a\}$ is a singleton, we simplify the notation $\Int_\cT\{a\}$, $\cl_\cT\{a\}$, $\mo_\cT\{a\}$ and $\opn_\cT\{a\}$ to $\Int_\cT a$, $\cl_\cT a$, $\mo_\cT a$ and $\opn_\cT a$.
When the topology $\cT$ is clear from the context, we drop the subscript $\cT$ in this notation.
Given a finite topological space $(X,\cT)$ we briefly write $X^{\operatorname{op}} := (X,\cTop)$ for the same space $X$ but with the opposite topology.

We recall that a subset $A$ of a topological space $X$
is {\em locally closed}
if every $x\in A$ admits a neighborhood $U$ in $X$ such that $A\cap U$ is closed in $U$.
Locally closed sets are important in the sequel.  
In particular, we have the following characterization of locally closed sets.

\begin{prop}
\label{prop:lcl}
(\cite[Problem 2.7.1]{En1989})
Assume $A$ is a subset of a topological space $X$. Then the following conditions are equivalent.
\begin{itemize}
   \item[(i)] $A$ is locally closed,
   \item[(ii)] $\mo_\cT A:=\cl_\cT A\setminus A$ is closed in $X$,
   \item[(iii)] $A$ is a difference of two closed subsets of $X$,
   \item[(iv)] $A$ is an intersection of an open set in $X$ and a closed set in $X$.
\end{itemize}
\end{prop}

As an immediate consequence of Proposition \ref{prop:lcl}(iv) we get the following three propositions.

\begin{prop}\label{prop:lcl-intersection}
  The intersection of a finite family of locally closed sets is locally closed.
\end{prop}

\begin{prop}\label{prop:lcl-cl-sub}
  If $A$ is locally closed and $B$ is closed, then $A\setminus B$ is locally closed.
\end{prop}

\begin{prop}\label{prop:lcl-op}
  Let $(X,\cT)$ be a finite topological space. A subset $A\subset X$ is locally closed in the topology $\cT$ if and only if it is locally closed in the topology $\cTop$.
\end{prop}

We recall that the topology $\cT$ is $T_2$
or {\em Hausdorff} if for any two different points $x,y\in X$, there exist disjoint sets
$U,V\in\cT$ such that $x\in U$ and $y\in V$.
It is $T_0$ or {\em Kolmogorov} if for any two different points $x,y\in X$ there exists a $U\in\cT$
such that $U\cap\{x,y\}$ is a singleton.

Finite topological spaces stand out from general topological spaces by the fact that the only Hausdorff topology
on a finite topological space $X$
is the discrete topology consisting of all subsets of $X$.

\begin{prop}\label{prop:cl-as-union}
Let $(X,\cT)$ be a finite topological space and $A\subset X$. Then 
  \[ \cl A =\bigcup_{a\in A}\cl a. \]
\end{prop}
\begin{proof}
  Let $A':=\bigcup_{a\in A}\cl a$. 
  Clearly, $A\subset A'\subset \cl A$. Since $X$ is finite, $A'$ is closed as a finite union of closed sets. Therefore, also $\cl A\subset A'$.
\end{proof}

A remarkable feature of finite topological spaces is the following
theorem.
\begin{thm}
  \label{thm:alexandroff}
  (P. Alexandrov, \cite{Al1937})
  For a preorder $\leq$ on a finite set $X$, there is a  topology $\cT_\leq$ on $X$
  whose open sets are the upper sets with respect to $\leq$.
  For a topology $\cT$ on a finite set $X$, there is a preorder $\leq_\cT$ where $x\leq_\cT y$ if and only if $x\in\cl_\cT y$. 
  The correspondences $\cT\mapsto\;\leq_\cT$ and $\leq\;\mapsto\cT_\leq$ are mutually inverse.
  Under these correspondences continuous maps are transformed into  order-preserving maps and vice versa.
  Moreover, the topology $\cT$ is $T_0$ (Kolmogorov)  if and only if the preorder $\leq_\cT$ is a partial order.
\end{thm}

The correspondence resulting from Theorem \ref{thm:alexandroff} provides a method to translate concepts and problems between topology and order theory in finite spaces.
In particular, closed sets are translated to down sets in this correspondence and we have the following straightforward proposition.

\begin{prop}\label{prop:cl-in-ftop}
 Let $(X,\cT)$ be a finite topological space. Then, for $A\subset X$ we have
  \begin{align*}
    \cl_\cT A &= \{x\in X\mid \exists_{a\in A}\ x\leq_\cT a\},\\
    \opn_\cT A &= \{x\in X\mid \exists_{a\in A}\ x\geq_\cT a\},\\
    \Int_\cT A &= \{a\in A\mid \forall_{x\in X}\ x\geq_\cT a\ \Rightarrow x\in A\}.
  \end{align*}
\end{prop}

In other words, 
$\cl_\cT A$ is the minimal down set with respect to $\leq_\cT$ containing $A$,
$\opn_\cT A$ is the minimal upper set with respect to $\leq_\cT$ containing $A$ and
$\Int_\cT A$ is the maximal upper set with respect to $\leq_\cT$ contained in $A$.

\begin{prop}
\label{prop:lcl-in-ftop}
Assume $X$ is a $T_0$ finite topological space and $A\subset X$.
Then $A$ is locally closed if and only if $A$ is convex with respect to $\leq_\cT$.
\end{prop}
\begin{proof}
  Assume that $A$ is locally closed. Let $x,y\in A$. 
  By Proposition \ref{prop:lcl} we can write $A=U\cap D$, where $U$ is open and $D$ is closed.  
  By Theorem \ref{thm:alexandroff} we know that $U$ is an upper set and $D$ is a down set with respect to $\leq_\cT$.
  Let $x,z\in A$ and let $y\in X$ be such that $x\leq_\cT y\leq_\cT z$. Since $x\in U$ and $U$ is an upper set, it follows, that $y\in U$. Since $z\in D$ and $D$ is a down set, it follows $y\in D$. Thus $y\in U\cap D=A$.

  Conversely, assume that $A$ is convex with respect to $\leq_\cT$.  
  By Proposition \ref{prop:lcl}(ii) it suffices to prove that $\mo_\cT A = \cl_\cT A\setminus A$ is closed. 
  Suppose the contrary. 
  Then there exist an $x\not\in\mo_\cT A$ and a $y\in\mo_\cT A$ such that $x \leq_\cT y$. 
  It follows from Proposition \ref{prop:cl-in-ftop} and $y \in \mo_\cT A \subset \cl_\cT A$ that there exists an element $z\in A$ such that $y \leq_\cT z$. 
  In consequence we get $x \leq_\cT z$, and therefore $x \in \cl_\cT A$. In view of
  $x\not\in\mo_\cT A$ this implies $x \in A$, and the assumed convexity of~$A$ then
  gives $y \in A$, which contradicts $y\in\mo_\cT A$. 
\end{proof}

For a $T_0$ finite topological space $(X,\cT)$ we define the associated abstract simplicial complex as the order complex of $(X,\leq_\cT)$. We denote it $\cK(X)$ and its geometric realization by $|\cK(X)|$.

\subsection{Homology of finite topological spaces}\label{sec:hom-ftop}

Given a topological space $X$ we denote by $H(X)$ the singular homology of $X$.
Note that singular homology is well-defined for all topological spaces.
In particular, it is defined for finite topological spaces. 

Let $(X,\cT)$ be a finite topological space. 
The McCord map $\mu_X: |\cK(X)|\mapsto X$ maps an $\alpha\in|\cK(X)|$ to the maximal element of the chain $\operatorname{supp}(\alpha)$. 
The Theorem of McCord \cite{MC1966} states that $\mu_X$ is a weak homotopy equivalence, that is, it induces isomorphisms in all homotopy groups.
In particular, $\mu$ induces an isomorphism ${\mu_X}_{\ast}: H(|\cK(X)|)\mapsto H(X)$ in singular homology. 
Moreover, there exists a chain map $\eta$
from simplicial chains of $\cK(X)$ to singular chains of $|\cK(X)|$ which
induces an isomorphism $\eta_\ast:H(\cK(X))\mapsto H(|\cK(X)|)$ between simplicial and singular homology \cite[Theorem 34.3]{Munkres1984}.
In particular
  \begin{align*}
    H(X)\cong H(|\cK(X)|)\cong H(\cK(X)).
  \end{align*}
For computational purposes this allows us to replace the singular homology of finite topological spaces by the simplicial homology of the associated simplicial complex.

Now, we recall some basic results from homology theory in the context of finite topological spaces.

\begin{prop}\label{prop:simplicial_rel_hom_of_posets}
Let $B\subset A$ be subsets of a finite topological space $X$. Then $\cK(B)$ is a subcomplex of $\cK(A)$ and
\[
H(A,B)\cong H(\cK(A),\cK(B)).
\]
\end{prop}

\begin{proof} The McCord map $\mu_X$ naturally induces a homomorphism ${\mu_X}_\ast(A,B)$ in relative homology. Consider the commutative diagram
\begin{center}
    \begin{footnotesize}
  \begin{tikzcd}[row sep=normal, column sep = 1.0em]  
  H_n(|\cK(B)|) \arrow[r]\arrow[d, "{\mu_X}_\ast"]  & 
  H_n(|\cK(A)|)\arrow[r]\arrow[d, "{\mu_X}_\ast"]  & 
  H_n(|\cK(A)|, |\cK(B)|) \arrow[r]\arrow[d, "{\mu_X}_\ast(A\text{,}B)"] & 
  H_{n-1}(|\cK(B)|) \arrow[r]\arrow[d, "{\mu_X}_\ast"] & 
  H_{n-1}(|\cK(A)|) \arrow[d, "{\mu_X}_\ast"]\\
H_n(B) \arrow[r]& 
  H_n(A) \arrow[r] & 
  H_n(A, B) \arrow[r] & 
  H_{n-1}(B) \arrow[r] & 
  H_{n-1}(A)
  \end{tikzcd}
  \end{footnotesize}
\end{center}
The Five Lemma \cite[Lemma 24.3]{Munkres1984} implies that ${\mu_X}_\ast(A,B)$ is also an isomorphism.
Similarly, the chain map $\eta$ induces a homomorphism $\eta_\ast(A, B)$. Thus again, the commutative diagram
\begin{center}
    \begin{footnotesize}
  \begin{tikzcd}[row sep=normal, column sep = 1.0em]
H_n(\cK(B)) \arrow[r]\arrow[d, "\eta_\ast"]  & 
  H_n(\cK(A))\arrow[r]\arrow[d, "\eta_\ast"]  & 
  H_n(\cK(A), \cK(B)) \arrow[r]\arrow[d, "\eta_\ast(A\text{,}B)"] & 
  H_{n-1}(\cK(B)) \arrow[r]\arrow[d, "\eta_\ast"] & 
  H_{n-1}(\cK(A)) \arrow[d, "\eta_\ast"]\\
  H_n(|\cK(B)|) \arrow[r] & 
  H_n(|\cK(A)|)\arrow[r] & 
  H_n(|\cK(A)|, |\cK(B)|) \arrow[r] & 
  H_{n-1}(|\cK(B)|) \arrow[r] & 
  H_{n-1}(|\cK(A)|) 
  \end{tikzcd}
  \end{footnotesize}
\end{center}
together with the Five Lemma implies that $\eta_\ast(A,B)$ is an isomorphism. It follows that ${\mu_X}_\ast(A,B)\circ\eta_\ast(A, B)$ is also an isomorphism.
\end{proof}

It follows from Proposition \ref{prop:simplicial_rel_hom_of_posets} that under its assumptions $H(A,B)$ is finitely generated. 
In particular, the $i$th \emph{Betti number} $\beta_i(A,B):=\rank H_i(A,B)$ is well-defined, as well as the \emph{Poincar\'{e} polynomial} 
\begin{equation}\label{eq:poincare_polynomial}
  p_{A,B}(t) := \sum_{i=1}^{\infty}\beta_i(A,B)t^i.
\end{equation}
In the sequel, we also need the finite counterpart of the excision theorem.

\begin{thm}\label{thm:simpl_excision}\cite[Theorem 9.1]{Munkres1984} (Excision theorem)  
  Let $K$ be a simplicial complex and let $K_0$ be its subcomplex. 
  Assume that $U$ is an open set contained in $|K_0|$ such that $|K|\setminus U$ is a polytope of a subcomplex $L$ of $K$ and $L_0$ is the subcomplex of $K$ whose polytope is $|K_0|\setminus U$. 
  Then the inclusion $(L,L_0)\hookrightarrow (K,K_0)$ induces an isomorphism 
\begin{equation*}
  H(L, L_0) \cong H(K,K_0)
\end{equation*}
in simplicial homology.
\end{thm}

\begin{thm}\label{thm:excision_ftop}
  Let $(X,\cT)$ be a finite topological space and let $A, B, C, D$ be closed subsets of $X$ such that $B\subset A$, $D\subset C$ and $A \setminus B = C \setminus D$. 
  Then $H(A, B) \cong H(C, D)$.
\end{thm}

\begin{proof}
  We first observe that $\cK (A) \setminus \cK (B) = \cK (C) \setminus \cK(D)$. Indeed, consider a chain $q$ in $A$ which is not a chain in $B$. 
  Let $q_0$ be the maximal element of $q$. 
  Then $q_0\not\in B$, because otherwise, since $B$ is a closed set, and therefore also a down set with respect to $\leq_\cT$, we get $q\subset B$.  
  Hence, $q_0\in A\setminus B = C\setminus D$. Since $C$ is a down set as a closed set, it follows that $q\subset C$ and clearly $q\not\subset D$.  
  Thus, $q\in\cK(C)\setminus\cK(D)$ which proves that $\cK(A)\setminus\cK(B)\subset\cK(C)\setminus\cK(D)$.  
  The proof of the opposite inclusion is analogous.

  Define $\breve{B}:=|\cK(A)|\setminus |\cK(\cl(A\setminus B))|$. 
  Clearly, $\breve{B}$ is open in $|\cK(A)|$.
  We will show that $\breve{B}\subset |\cK(B)|$.
  Let $\alpha\in\breve{B}$. Set $r:=\supp{(\alpha)}$ and $r_0:=\max(r)$. 
  Suppose that $r_0\not\in B$. Then, $r_0\in A\setminus B$ and $r\subset \cl(A\setminus B)$ which implies 
  $\alpha\in|r|\subset|\cK(\cl(A\setminus B))|$, a contradiction.
  Hence, $r\subset B$ and $\alpha\in|r|\subset|\cK(B)|$. 
  Moreover,
  \begin{align*}
    |\cK(A)|\setminus\breve{B} &= |\cK(A)|\setminus \big( |\cK(A)|\setminus |\cK(\cl(A\setminus B))|\big) = |\cK(\cl(A\setminus B))|
  \end{align*}
  and by Proposition \ref{prop:order-complex-oper}
  \begin{align*}
    |\cK(B)|\setminus\breve{B} &= |\cK(B)|\cap|\cK(\cl(A\setminus B))|=
    |\cK(\cl(A\setminus B)\setminus(A\setminus B))| \\
    &= |\cK(\mo(A\setminus B))|.
  \end{align*}
  Analogous properties hold for $\breve{D}:=|\cK(C)|\setminus |\cK(\cl(C\setminus D))|$ in $|\cK(C)|$.
  Therefore, by Theorem \ref{thm:simpl_excision} we have the following isomorphisms
  \begin{align*}
    H(\cK(A), \cK(B)) &\cong H(\cK(\cl(A\setminus B), \cK(\mo(A\setminus B)), \\
    H(\cK(C), \cK(D)) &\cong H(\cK(\cl(C\setminus D), \cK(\mo(C\setminus D)).
  \end{align*}
  Note that according to $A\setminus B = C\setminus D$ we have
  $\cK (\cl(A\setminus B)) =\cK(\cl(C\setminus D))$ and
  $\cK (\mo(A\setminus B)) =\cK(\mo(C\setminus D))$. 
  Thus, with Proposition \ref{prop:simplicial_rel_hom_of_posets} we get
  \begin{align*}
  H\left(A, B\right) &\cong H(\cK(A), \cK(B)) \cong 
    H(\cK(\cl(A\setminus B), \cK(\mo(A\setminus B))\\
    &= H(\cK(\cl(C\setminus D), \cK(\mo(C\setminus D))\cong H(\cK(C), \cK(D))
    \cong H\left(C, D\right),
  \end{align*}
  which completes the proof of the theorem.
\end{proof}

\begin{thm}\label{thm:exact-hom-seq-triple}\cite[Chapter 24]{Munkres1984} 
Let $B\subset A\subset X$ be a triple of topological spaces. The inclusions induce the following exact sequence, called the \emph{exact homology sequence of the triple}:
  \begin{align*}
    \dotsc\rightarrow H_n(A,B)\rightarrow H_n(X,B)\rightarrow H_n(X,A)\rightarrow H_{n-1}(A,B)\rightarrow\dotsc.
  \end{align*}
\end{thm}

\begin{thm}\label{thm:relMV-simplicial}\cite[Chapter 25 Ex.2]{Munkres1984}(Relative simplicial Mayer-Vietoris sequence)
Let $K$ be a simplicial complex.  
Assume that $K_0$ and $K_1$ are subcomplexes of $K$ such that $K=K_0\cup K_1$ and $L_0$ and $L_1$ are subcomplexes of $K_0$ and $K_1$, respectively. Then there is an exact sequence
  \begin{align*}
    \dotsc\rightarrow H_n(K_0\cap K_1, L_0\cap L_1)\rightarrow H_n(K_0,L_0)\oplus H_n(K_1,L_1)\rightarrow \\
    H_n(K_0\cup K_1, L_0\cup L_1) \rightarrow H_{n-1}(K_0\cap K_1, L_0\cap L_1)\dotsc,
  \end{align*}
  called the \emph{relative Mayer-Vietoris sequence}.
\end{thm}

\begin{thm}\label{thm:relMV-ftop}(Relative Mayer-Vietoris sequence for finite topological spaces)
Let $X$ be a finite topological space.
Assume that $Y_0\subset X_0, Y_1\subset X_1$ are pairs of closed sets in $X$ such that
$X=X_0\cup X_1$. Then there is an exact sequence 
  \begin{align*}
    \dotsc\rightarrow H_n(X_0\cap X_1, Y_0\cap Y_1)\rightarrow H_n(X_0,Y_0)\oplus H_n(X_1,Y_1)\rightarrow \\
    H_n(X_0\cup X_1, Y_0\cup Y_1)\rightarrow H_{n-1}(X_0\cap X_1, Y_0\cap Y_1)\dotsc.
  \end{align*}
\end{thm}

\begin{proof}
  By Proposition \ref{prop:order-complex-oper} we have $\cK(X) = \cK(X_0)\cup\cK(X_1)$ and $\cK(Y_0\cup Y_1) = \cK(Y_0)\cup\cK(Y_1)$.
  Thus, the proof follows from the relative simplicial Mayer-Vietoris Theorem \ref{thm:relMV-simplicial} and Proposition \ref{prop:simplicial_rel_hom_of_posets}.
\end{proof}


\section{Dynamics of combinatorial multivector fields}
\label{sec:dcmvf}

\subsection{Multivalued dynamical systems in finite topological spaces}

By a \emph{combinatorial dynamical system} or briefly, a \emph{dynamical system} in a finite topological space $X$ we mean a multivalued map $\Pi:X\times\ZZ^+ \multimap X$ such that
  \begin{equation}\label{eq:cds}
    \Pi\left(\Pi(x,m),n\right)=\Pi(x,m+n)
    \quad\mbox{for all}\quad
    m,n \in \ZZ^+, \; x \in X.
  \end{equation}
Typically, one also assumes that $\Pi$ is continuous in some sense but we do not need such an assumption in this paper.

Let $\Pi$ be a combinatorial dynamical system. Consider the multivalued map $\Pi^n:X\multimap X$ given by $\Pi^n(x):=\Pi(x,n)$. We call $\Pi^1$ the \emph{generator\/} of the dynamical system $\Pi$.  
It follows from (\ref{eq:cds}) that the combinatorial dynamical system $\Pi$ is uniquely determined by its generator.
Thus, it is natural to identify a combinatorial dynamical system with its generator.
In particular, we consider any multivalued map $\Pi: X\multimap X$ as a combinatorial dynamical system $\Pi: X\times \ZZ^+ \multimap X$ defined recursively by
  \begin{align*}
    \Pi(x,1)&:=\Pi(x),\\
    \Pi(x,n+1)&:=\Pi(\Pi(x, n)),
  \end{align*}
as well as $\Pi(x,0) := x$.
We call it the combinatorial dynamical system induced by a map $\Pi$.
In particular, the inverse $\Pi^{-1}$ of $\Pi$ also induces a combinatorial dynamical system. We call it the \emph{dual dynamical system}.

\subsection{Solutions and paths}\label{subs:solutions}

By a $\ZZ$-interval we mean a set of the form $\ZZ\cap I$ where $I$ is an interval in $\RR$.
A $\ZZ$-interval is {\em left bounded} if it has a minimum; otherwise it is \emph{left-infinite}.
It is {\em right bounded} if it has a maximum; otherwise it is \emph{right-infinite}.
It is {\em bounded} if it has both a minimum and a maximum.
It is {\em unbounded} if it is not bounded.

A {\em solution} of a combinatorial dynamical system $\Pi:X\multimap X$ in $A\subset X$ is a partial map $\varphi:\ZZ\pto A$ whose {\em domain},
denoted $\dom \varphi$,  is a $\ZZ$-interval and for any $i,i+1\in\dom\varphi$ the inclusion $\varphi(i+1)\in \Pi(\varphi(i))$ holds.
The solution {\em passes} through $x\in X$ if $x=\varphi(i)$ for some $i\in\dom\varphi$.
The solution $\varphi$ is {\em full} if $\dom \varphi=\ZZ$.
It is a {\em backward solution} if $\dom \varphi$ is left-infinite.
It is a {\em forward solution} if $\dom \varphi$ is right-infinite.
It is a {\em partial solution} or simply a {\em path} if $\dom \varphi$ is bounded.

A full solution $\varphi:\ZZ\rightarrow X$ is \emph{periodic} if there exists a $T\in\NN$ such that $\varphi(t+T) =\varphi(t)$ for all $t\in\ZZ$.
Note that every closed path may be extended to a periodic solution.

If the maximum of $\dom\varphi$ exists, we call the value of $\varphi$ at this maximum the {\em right endpoint} of $\varphi$.
If the minimum of $\dom\varphi$ exists, we call the value of $\varphi$ at this minimum the {\em left endpoint} of $\varphi$.
We denote the left and right endpoints of $\varphi$, respectively, by $\lep{\varphi}$ and $\rep{\varphi}$.

By a {\em shift} of a solution $\varphi$ we mean the composition
$\varphi\circ\tau_n$, where the map $\tau_n:\ZZ\ni m\mapsto m+n\in \ZZ$ is translation.
Given two solutions $\varphi$ and $\psi$ such that $\lep{\psi}$ and $\rep{\varphi}$ exist and $\lep{\psi}\in\Pi(\rep{\varphi})$,
there is a unique shift $\tau_n$ such that $\varphi\cup\psi\circ\tau_n$ is a solution.
We call this union of paths the {\em concatenation} of $\varphi$ and $\psi$ and we denote it by $\varphi\cdot\psi$.   
We also identify each $x\in X$ with the trivial solution $\varphi:\{0\}\rightarrow \{x\}$.
For a full solution $\varphi$ we denote the restrictions $\left. \varphi \right|_{\ZZ^+}$ by $\varphi^+$ and $\left. \varphi \right|_{\ZZ^-}$ by $\varphi^-$.
We finish this section with the following straightforward proposition.

\begin{prop}
  If $\varphi:\ZZ\rightarrow X$ is a full solution of a dynamical system $\Pi:X\multimap X$, then $\ZZ\ni t\rightarrow\varphi(-t)\in X$ is a solution of the dual dynamical system induced by $\Pi^{-1}$. 
  We call it the \emph{dual solution} and denote it $\varphi^{\op}$.
\end{prop}

\subsection{Combinatorial multivector fields}
\label{sec:cmvf}

Combinatorial multivector fields on Lefschetz complexes were introduced in \cite[Definition 5.10]{Mr2017}.
In this paper, we generalize this definition as follows.
Let $(X,\cT)$ be a finite topological space.
By a {\em combinatorial multivector} in $X$ we mean a locally closed and nonempty subset of $X$.
We define a {\em combinatorial multivector field} as a partition $\cV$ of $X$ into multivectors.
Therefore, unlike \cite{Mr2017}, we do not assume that a multivector has a unique maximal element with respect to $\leq_\cT$.
The unique maximal element assumption was introduced in \cite{Mr2017} for technical reasons but
it is very inconvenient in applications. As an example we mention the following straightforward proposition 
which is not true in the setting of \cite{Mr2017}.
\begin{prop}
\label{prop:subfield}
  Assume $\cV$ is a combinatorial multivector field on a finite topological space $X$ and $Y\subset X$ is a locally closed
  subspace. Then 
\[
   \cV_Y:=\setof{V\cap Y\mid V\in\cV, V\cap Y\neq\emptyset}
\]
is a multivector field in $Y$. We call it the {\em multivector field induced by $\cV$ on $Y$}.
\qed
\end{prop}

We say that a multivector $V$ is {\em critical} if the relative singular homology $H(\cl V,\mouth V)$  is non-zero.
A multivector $V$ which is not critical is called {\em regular}.
For each $x\in X$ we denote by $\vclass{x}$ the unique multivector in $\cV$ which contains~$x$.
If the multivector field $\cV$ is clear from the context, we write briefly $[x]:=\vclass{x}$.  

We say that $x\in X$ is \emph{critical} (respectively \emph{regular}) with respect to $\cV$ if $\vclass{x}$ is critical (respectively regular).
We say that a subset $A\subset X$ is $\cV$-{\em compatible} if for each $x\in X$ either $\vclass{x}\cap A=\emptyset$ or $\vclass{x}\subset A$.
Note that every $\cV$-compatible set $A\subset X$ induces a well-defined multivector field $\cV_A:=\{V\in\cV\mid V\subset A\}$ on $A$.
The next proposition follows immediately from the definition of a $\cV$-compatible set.
\begin{prop}\label{prop:union-intersection-vcomp}
  The union and the intersection of a family of $\cV$-compatible sets is $\cV$-compatible.
\end{prop}

We associate with every multivector field $\cV$ a combinatorial dynamical system on~$X$ induced by the multivalued map $\Pi_\cV: X\mto X$ given by
\begin{equation}\label{eq:piv}
   \Piv(x):= \vclass{x}\cup\cl x.
\end{equation}
The following proposition is straightforward.
\begin{prop}
  Let $\cV$ be a multivector field on $X$. Then 
  \[ \Piv(x) = \vclass{x}\cup\mo x. \]
\end{prop}

\begin{figure}
  \includegraphics[width=0.5\textwidth]{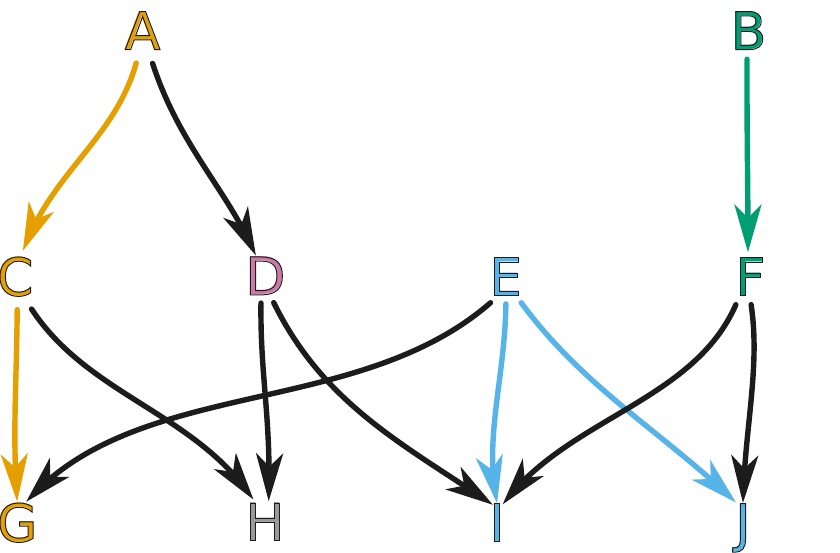}  
  \caption{An example of a combinatorial  multivector field 
  $\cV=\{\{A,C,G\}, \{D\}, \{H\}, \{E,I,J\}, \{B,F\}\}$ 
  on a finite topological space consisting of ten points. There are two regular multivectors, $\{A,C,G\}$ and $\{E,I,J\}$, the others are critical. Both the nodes and the connecting edges of each multivector are highlighted with a different color.} \label{fig:poset-mvf}
\end{figure}

We have following proposition.
\begin{prop}\label{prop:preimage-piv-opn}
  Let $\cV$ be a combinatorial multivector field on $(X,\cT)$. If $A\subset X$, then
  \begin{align*}
    \Piv^{-1}(A)=\bigcup_{x\in A} [x]_\cV\cup \opn x.
  \end{align*}
\end{prop}
\begin{proof}
  Assume $y\in\Piv^{-1}(A)$. By (\ref{eq:large_premage}) there exists an $x\in A$ such that $x\in \Piv(y)$, that is, $x\in\cl y\cup[y]_\cV=\Piv(y)$.
  If $x \in\cl y$, then from Proposition \ref{thm:alexandroff} we have $x\leq_\cT y$. It follows by \ref{prop:cl-in-ftop} that $y\in\opn x$. The case when $x\in [y]_\cV$ is trivial. 
  Hence, $y\in \opn x\cup [x]_\cV$ and consequently 
  \[\Piv^{-1}(A)\subset \bigcup_{x\in A}[x]_\cV\cup\opn x.\]
  
  In order to show the opposite inclusion consider an $x\in A$ and a $y\in\opn x\cup [x]_\cV$.
  If $y\in[x]_\cV$, then clearly $x\in[y]_\cV\subset\Piv(y)$ which implies $x\in\Piv(y)\cap A\neq\emptyset$.
  Thus $y\in\Piv^{-1}(A)$. 
  If $y\in\opn{x}$, then by Proposition \ref{prop:cl-in-ftop} we have $x\leq_\cT y$ and therefore $x\in\cl y$. 
  Thus, $x\in\Piv(y)$ and again $\Piv(y)\cap A\neq\emptyset$. 
  Hence, $y\in\Piv^{-1}(A)$ which completes the proof of the opposite inclusion. 
\end{proof}

\begin{figure}
  \includegraphics[width=0.2\textwidth]{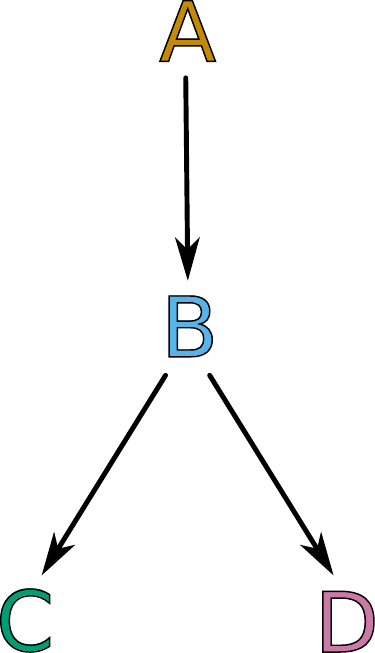}
  \hspace{2cm}
  \includegraphics[width=0.2\textwidth]{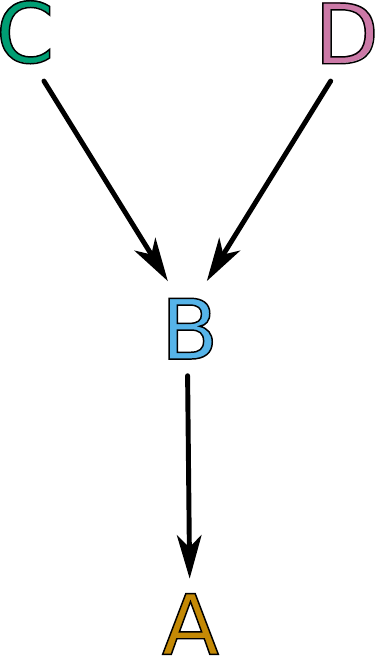}
  \caption{
    An example of a finite topological space $X$ and $\Xop$ consisting of four points and with the same partition into multivectors $\cV=\{\{A\}, \{B\}, \{C\}, \{D\}\}$. 
    In $X$ multivectors $\{B\}$, $\{C\}$ and $\{D\}$ are critical, while in $\Xop$ only $\{A\}$
    is critical.
  }\label{fig:non-dual-crit}
\end{figure}

Note that by Proposition \ref{prop:lcl-op} a multivector in a finite topological space $X$ is also a multivector in $\Xop$, that is, in the space $X$ with the opposite topology.
Thus, a multivector field $\cV$ in $X$ is also a multivector field in $\Xop$.
However, the two multivector fields cannot be considered the same, because the change in topology implies the change of the location of critical and regular multivectors (see Figure \ref{fig:non-dual-crit}).
We indicate this in notation by writing $\cVop$ for the multivector field $\cV$ considered with the opposite topology.

The multivector field $\cVop$ induces a combinatorial dynamical system $\Pi_{\cVop}:\Xop\multimap\Xop$ given by $\Pi_{\cVop}(x):=[x]_\cV\cup\cl_{\cTop} x$.
As an immediate consequence of Proposition \ref{prop:dual-top} and Proposition \ref{prop:preimage-piv-opn} we get following result.
\begin{prop}\label{prop:piv-dual}
  The combinatorial dynamical system $\Pi_\cVop$ is dual to the combinatorial dynamical system $\Pi_\cV$,
  that is, we have $\Pi_\cVop =\Pi_\cV^{-1}$.
\end{prop}

\subsection{Essential solutions}
Given a multivector field $\cV$ on a finite topological space $X$ by a solution (full solution, forward solution, backward solution, partial solution or path) of $\cV$ we mean a corresponding solution type of the combinatorial dynamical system $\Pi_\cV$.
Given a solution $\varphi$ of $\cV$ we denote by $\cV(\varphi)$ the set of multivectors $V\in\cV$ such that $V\cap\im\varphi\neq\emptyset$.
We denote the set of all paths of $\cV$ in a set $A$ by $\paths_\cV(A)$ and define
  \begin{align*}
    \pathsv(x, A) &:= \{\varphi\in\pathsv(A)\ \mid \ \varphi(0)=x\},\\
    \pathsv(x, y, A) &:= \{\varphi\in\pathsv(A)\ \mid\ \lep{\varphi}=x\ \text{and}\ \rep{\varphi}=y\}.
  \end{align*}
We denote the set of full solutions of $\cV$ in $A$ (respectively backward or forward solutions in $A$) by $\sol_\cV(A)$ (respectively $\sol_\cV^-(A)$, $\sol_\cV^+(A)$).  
We also write
\begin{align*}
\sol_\cV(x,A) &:= \{\varphi\in\sol_\cV(A)\ \mid \ \varphi(0)=x\}.
\end{align*}
Observe that by (\ref{eq:piv}) $x\in\Piv(x)$ for every $x\in X$.
Hence, a constant map from an interval to a point is always a solution. 
This means that every solution can easily be extended to a full solution. 
In consequence, every point is recurrent which is not typical.
To remedy this we introduce the concept of an essential solution.

A full solution $\varphi:\ZZ\to X$ is {\em left-essential} (respectively {\em right-essential}) if for every regular $x\in \im\varphi$ the set $\setof{t\in\ZZ\mid\varphi(t)\not\in\vclass{x}}$ is {\em left-infinite} (respectively {\em right-infinite}).
We say that $\varphi$ is {\em essential} if it is both left- and right-essential.
We say that a point $x\in X$ is \emph{essentially recurrent} if an essential periodic solution passes through $x$.
Note that a periodic solution $\varphi$ is essential either if $\#\cV(\varphi)\geq 2$ or if the unique multivector in $\cV(\varphi)$ is critical.

We denote the set of all essential solutions in $A\subset X$ (respectively left- or right-essential solutions in $A$) by $\esol_\cV(A)$ (respectively $\esolp(A)$, $\esolm(A)$) and   
  the set of all essential solutions in a set $A\subset X$ passing through a point~$x$ by
\begin{align*}
  \esol_\cV(x,A) &:= \{\varphi\in\esol(A)\ \mid \ \varphi(0)=x\}
\end{align*}
and we define the \emph{invariant part} of $A\subset X$ by
\begin{equation}\label{eq:inv_set}
  \inv_\cV A := \left\{ x\in A\ |\ \esol(x,A)\neq\emptyset \right\}.
\end{equation}
In particular, if $\inv_\cV A = A$ then we say that $A$ is an \emph{invariant set} for $\cV$.
We drop the subscript $\cV$ in $\sol_\cV$, $\esol_\cV$ and $\inv_\cV$ whenever $\cV$ is clear from the context.

\begin{prop}\label{prop:union-of-invariant-sets}
  Let $A,B\subset X$ be invariant sets. Then $A\cup B$ is also an invariant set.
\end{prop}
\begin{proof}
  Let $x\in A$. 
  By the definition of an invariant set there exists an essential solution  $\varphi\in\esol(x,A)$. 
  It is clear that $\varphi$ is also an essential solution in $A\cup B$. 
  Thus $\esol(x,A\cup B)\neq\emptyset$. 
  The same holds for $B$. 
  Hence $\inv(A\cup B)=A\cup B$.
\end{proof}

\subsection{Isolated invariant sets}
In this subsection we introduce the combinatorial counterpart of the concept of an isolated invariant set. 
In order to emphasize the difference, we say that an isolated invariant set is isolated by an isolating set, not by an isolating neighborhood.
In comparison to the classical theory of dynamical systems, the crucial difference is that we cannot guarantee the existence of disjoint isolating sets for two disjoint isolated invariant sets. 
This is caused by the tightness of the finite topological space. 

\begin{defn}\label{def:iso-inv-set}
  A closed set $N$ \emph{isolates} an invariant set $S\subset N$, if the following two conditions hold:
  \begin{enumerate}[(a)]
    \item Every path in $N$ with endpoints in $S$ is a path in $S$,
    \item $\Pi_\cV(S)\subset N$.
  \end{enumerate}
  In this case, we also say that $N$ is an \emph{isolating set} for $S$.  
  An invariant set~$S$ is \emph{isolated} if there exists a closed set $N$ meeting the above conditions.
\end{defn}

An important example is given by the following straightforward proposition. 
\begin{prop}
\label{prop:invX}
The whole space $X$ isolates its invariant part $\Inv X$. In particular, $\Inv X$ is an isolated invariant set.
\qed
\end{prop}

\begin{prop}\label{prop:iso-is-vcomp}
  If $S\subset X$ is an isolated invariant set, then $S$ is $\cV$-compatible.
\end{prop}
\begin{proof}
  Suppose the contrary. Then there exists an $x\in S$ and a $y\in[x]_\cV\setminus S$.  
  Let $N$ be an isolating set for $S$. 
  It follows from Definition \ref{def:iso-inv-set}(b), that $y\in\Piv(x)\subset N$. 
  It is also clear that $x\in\Piv(y)$. 
  Thus the path $x\cdot y\cdot x$ is a path in $N$ with endpoints in $S$,
  but it is not contained in $S$, and this in turn contradicts Definition \ref{def:iso-inv-set}(a).
\end{proof}


The finiteness of the space allows us to construct the smallest possible isolating set. 
More precisely, we have the following straightforward proposition.

\begin{prop}\label{prop:smallest_iso_neigh}
Let $N$ be an isolating set for an isolated invariant set $S$. If $M$ is a closed set such that $S\subset M\subset N$, then $S$ is also isolated by $M$. In particular, $\cl S$ is the smallest isolating set for $S$.\qed
\end{prop}

\begin{prop}\label{prop:iso_inv_is_conv}
Let $S \subset X$.
If $S$ is an isolated invariant set, then $S$ is locally closed.
\end{prop}
\begin{proof}
By Proposition \ref{prop:smallest_iso_neigh} the set $N := \cl S$ is an isolating set for $S$. 
Assume that $S$ is not locally closed. 
By Proposition \ref{prop:lcl-in-ftop} there exist $x, z \in S$ and a $y \not\in S$ such that $x \leq_\cT y \leq_\cT z$. 
Hence, it follows from Theorem \ref{thm:alexandroff} that $x\in\cl_\cT y$ and $y\in\cl_\cT z$.  
In particular, $x,y,z\in\cl S$. 
It follows that $\varphi:=z\cdot y\cdot x$ is a solution in $\cl S$ with endpoints in $S$. 
In consequence, $y\in S$, a contradiction.
\end{proof}

In particular, it follows from Proposition~\ref{prop:iso_inv_is_conv} that if $S$ is an isolated invariant set, 
then we have the induced multivector field $\cV_S$ on $X$.

\begin{prop}\label{prop:inv_conv_vcomp_is_iso}
Let $S$ be a locally closed, $\cV$-compatible invariant set. Then $S$ is an isolated invariant set.
\end{prop}
\begin{proof}
  Assume that $S$ is a $\cV$-compatible and locally closed invariant set. We will show that $N:=\cl S$ isolates $S$. 
  We have 
  \[
  \Piv(S) = \bigcup_{x\in S} \cl x \cup \bigcup_{x\in S} [x]_\cV = \cl S \cup S = \cl S \subset N.
  \]
  Therefore condition (b) of Definition \ref{def:iso-inv-set} is satisfied.

  We will now show that every path in $N$ with endpoints in $S$ is a path in $S$.
  Let $\varphi:=x_0\cdot x_1\cdot...\cdot x_n$ be a path in $N$ with endpoints in $S$.
  Thus, $x_0, x_n\in S$. Suppose that there is an $i\in\{0,1,...,n\}$ such that $x_i\not\in S$. 
  Without loss of generality we may assume that $i$ is maximal such that $x_i\not\in S$. 
  Then $x_{i+1}\neq x_i$ and $i<n$, because $x_n\in S$. 
  We have $x_{i+1}\in \Pi_\cV(x_i)=[x_i]_\cV\cup\cl x_i$.
  Since $x_i\not\in S$, $x_{i+1}\in S$ and $S$ is $\cV$-compatible, we cannot have $x_{i+1}\in[x_i]_\cV$. 
  Therefore, $x_{i+1}\in\cl x_i$.
  Since $\varphi$ is a path in $N=\cl{S}$, we have $x_i\in\cl S$.
  Hence, $x_i\in\cl z$ for a $z\in S$.
  It follows from Proposition \ref{prop:lcl-in-ftop} that $x_i\in S$, because 
$x_{i+1},z\in S$, $x_{i+1}\in\cl x_i$, $x_i\in\cl z$ and $S$ is locally closed.
Thus, we get a contradiction proving that also condition (a) of Definition \ref{def:iso-inv-set} is satisfied.
  In consequence, $N$ isolates $S$ and $S$ is an isolated invariant set.
\end{proof}

\subsection{Multivector field as a digraph}

Let $\cV$ be a multivector field in $X$. 
We denote by $G_\cV$ the multivalued map $\Piv$ interpreted as a digraph.
%

\begin{prop}
  Assume $A\subset X$ is strongly connected in $G_\cV$. Then the following conditions are pairwise equivalent.
  \begin{enumerate}[(i)]
    \item There exists an essentially recurrent point~$x$ in~$A$, that is,
          there exists an essential periodic solution in~$A$ through~$x$,
    \item $A$ is non-empty and every point in~$A$ is essentially recurrent in~$A$,
    \item $\inv A\neq\emptyset$.
  \end{enumerate}
\end{prop}
\begin{proof}
  Assume (i). Then $A\neq\emptyset$.
  Let $x\in A$ be an essentially recurrent point in~$A$ and let $y\in A$ be arbitrary.
  Since $A$ is strongly connected, we can find a periodic solution $\varphi$ in~$A$ passing through~$x$ and~$y$.
  If $\#\cV(\varphi)\geq 2$, then $\varphi$ is essential and $y$ is essentially recurrent in~$A$.
  Otherwise $[x]_\cV=[y]_\cV$ and we may easily modify the essential periodic solution in~$A$ through $x$ to an essential periodic solution in~$A$ through~$y$.
  This proves (ii).
  Implication (ii)$\Rightarrow$(iii) is straightforward.
  To prove that (iii) implies (i) assume that $\varphi$ is an essential solution in $A$,
  If $\#\cV(\varphi)=1$, then the unique multivector $V\in\cV(\varphi)$ is critical and every $x\in V\subset A$ is essentially recurrent in~$A$.
  Otherwise we can find points $x,y\in A$ such that $[x]_\cV\neq[y]_\cV$.
  Since $A$ is strongly connected, we can find paths $\psi_1\in\pathsv(x,y,A)$ and $\psi_2\in\pathsv(y,x,A)$.
  Then $\psi_1\cdot\psi_2$ extends to an essential periodic solution in~$A$ through $x$ proving that $x\in A$ is essentially recurrent in~$A$.
\end{proof}

The above result considered the situation of a strongly connected set in~$G_\cV$.
If in addition we assume that this set is maximal, that is, a strongly connected
component in~$G_\cV$, one obtains the following result.

\begin{prop}\label{prop:scc-Vcomp-lcl}
  Let $\cV$ be a multivector field on $X$ and let $G_\cV$ be the associated digraph.  
  If $C\subset X$ is a strongly connected component of $G_\cV$, then $C$ is $\cV$-compatible and locally closed.
\end{prop}
\begin{proof}
  Let $x\in C$ and $y\in\vclass{x}$. It is clear that $x\cdot y\in\pathsv(x,y,X)$ and $y\cdot x\in\pathsv(y,x,X)$.  
  Hence $C$ is $\cV$-compatible.

  Let $x,z\in C$, $y\in X$ be such that $x\leq_\cT y\leq_\cT z$. 
  Since $C$ is strongly connected we can find a path $\rho$ from $x$ to $z$. 
  Clearly, by Proposition \ref{prop:cl-in-ftop} and (\ref{eq:piv}) we have $y\in\Piv(z)$ and $x\in\Piv(y)$. 
  Thus $y\cdot\rho\in\pathsv(y,z,X)$ and $z\cdot y\in\pathsv(z,y,X)$. 
  It follows that $y\in C$.
  Hence, $C$ is convex and, by Proposition \ref{prop:lcl-in-ftop}, C is locally closed.
\end{proof}

\begin{thm}\label{thm:scc_is_iso_inv}
  Let $\cV$ be a multivector field on $X$ and let $G_\cV$ be the associated digraph. 
  If $C\subset X$ is a strongly connected component of $G_\cV$ such that $\esol(C)\neq\emptyset$, then $C$ is an isolated invariant set.
\end{thm}
\begin{proof}
  According to Proposition \ref{prop:inv_conv_vcomp_is_iso} it suffices to prove that $C$ is a $\cV$-compatible, locally closed invariant set.
  It follows from Proposition \ref{prop:scc-Vcomp-lcl} that $C$ is $\cV$-compatible and locally closed.
  Thus, we only need to show that $C$ is invariant.
  Since $\inv C\subset C$, we only need to prove that $C\subset\inv C$.
  Let $y\in C$.
  Since $\esol(C)\neq\emptyset$, we may take an $x\in C$ and a $\varphi\in\esol(x,C)$.
  Since $C$ is strongly connected we can find paths $\rho$ and $\rho'$ in $C$ from $x$ to $y$ and from $y$ to $x$ respectively. 
  Then the solution $\varphi^-\cdot\rho\cdot\rho'\cdot\varphi^+$ is a well-defined essential solution through $y$ in $C$. 
  Thus, $\esol(y,C)\neq\emptyset$, which proves that we have $y\in\inv C$. 
\end{proof}

\begin{figure}
  \includegraphics[width=0.5\textwidth]{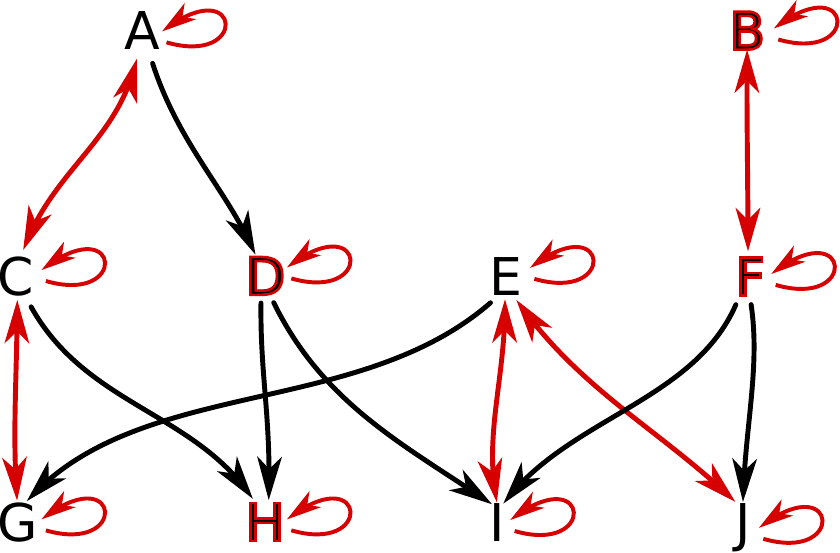}
  \caption{Digraph $G_\cV$ for a multivector field from Figure~\ref{fig:poset-mvf}.   
Black edges are induced by closure relation, while the red bi-directional edges represent connections within a multivector. 
  For clarity, we omit the edges that can be obtained by the between-level transitivity (e.g., from $A$ to $G$). Nodes that are part of a critical multivector are additionally bolded in red.
  }
\end{figure}

\section{Index pairs and Conley index} \label{sec:index-pairs-conley-index}

In this section we construct the Conley index of an isolated invariant set of a combinatorial multivector field.
As in the classical case we define index pairs, prove their existence and prove that the homology of an index pair depends only on the isolated invariant set and not on the choice of index pair.

\subsection{Index pairs and their properties}
\begin{defn}\label{def:index_pair}
Let $S$ be an isolated invariant set. 
A pair $P=(P_1, P_2)$ of closed subsets of $X$ such that $P_2\subset P_1$, is called an \emph{index pair for $S$} if
\begin{enumerate}[label={(IP\arabic*)}]
  \item\label{def:index_pair_cond_1} 
    $x \in P_2,\ y \in \Pi_\cV(x) \cap P_1\ \Rightarrow\ y \in P_2$ (positive invariance),\label{it:ind_pair_1}
  \item $x \in P_1,\ \Pi_\cV(x) \setminus P_1 \neq \emptyset\ \Rightarrow\ x \in P_2$ (exit set),\label{it:ind_pair_2}
  \item $S = \Inv (P_1\setminus P_2)$ (invariant part).\label{it:ind_pair_3}
\end{enumerate}
An index pair $P$ is said to be \emph{saturated} if $S=P_1\setminus P_2$.
\end{defn}

\begin{prop}\label{prop:P1_isolates}
  Let $P$ be an index pair for an isolated invariant set $S$. Then $P_1$ isolates $S$.
\end{prop}
\begin{proof}
    According to our assumptions, the set~$P_1$ is closed, and we clearly have
    $S=\Inv (P_1\setminus P_2)\subset P_1\setminus P_2\subset P_1$. Thus, it only
    remains to be shown that conditions~(a) and~(b) in Definition~\ref{def:iso-inv-set}
    are satisfied.
    
  Suppose there exists a path $\psi:=x_0\cdot x_1\cdot\dotsc\cdot x_n$ in $P_1$ such that $x_0, x_n\in S$ and $x_i\in P_1\setminus S$ for some $i\in\{1,2,\dotsc,n-1\}$.
  First, we will show that $\im\psi\subset P_1\setminus P_2$. 
  To this end, suppose the contrary. 
  Then, there exists an $i\in\{1,2,\dotsc,n-1\}$ such that $x_i\in P_2$ and $x_{i+1}\in P_1\setminus P_2$.
  Since $\psi$ is a path we have $x_{i+1}\in\Pi_\cV(x_i)$.
  But, \ref{it:ind_pair_1} implies $x_{i+1}\in P_2$, a contradiction.

  Since $S$ is invariant and $x_0,x_n\in S$, we may take a $\varphi_0\in\esol(x_0, S)$ and a $\varphi_n\in\esol(x_n, S)$. 
  The solution $\varphi_0^-\cdot\psi\cdot\varphi_n^+$ is an essential solution in $P_1\setminus P_2$ through $x_i$. 
  Thus, $x_i\in\Inv(P_1\setminus P_2)=S$, a contradiction.
  This proves that every path in $P_1$ with endpoints in $S$ is contained in $S$, 
  and therefore Definition~\ref{def:iso-inv-set}(a) is satisfied.
  
  In order to verify~(b), let $x \in S$ be arbitrary. We have already seen that then
  $x \in P_1 \setminus P_2 \subset P_1$. Now suppose that $\Pi_\cV(x) \setminus P_1
  \neq \emptyset$. Then~\ref{it:ind_pair_2} implies $x \in P_2$, which contradicts
  $x \in P_1 \setminus P_2$. Therefore, we necessarily have $\Pi_\cV(x) \setminus P_1
  = \emptyset$, that is, $\Pi_\cV(x) \subset P_1$, which immediately implies~(b).
  Hence, $P_1$ isolates $S$.
\end{proof}

One can easily see from the above proof that in Definition~\ref{def:index_pair}
one does not have to assume that~$S$ is an isolated invariant set. In fact, the
proof of Proposition~\ref{prop:P1_isolates} implies that any invariant set which
admits an index pair is automatically an isolated invariant set. Furthermore, the
following result shows that every isolated invariant set~$S$ does indeed admit at
least one index pair.

\begin{prop}\label{prop:minimal_index_pair}
Let $S$ be an isolated invariant set. Then $\left(\cl S, \mo S\right)$ is a saturated index pair for $S$.
\end{prop}
\begin{proof}
  To prove \ref{it:ind_pair_1} assume that $x\in\mo S$ and $y\in\Pi_\cV(x)\cap\cl S$. 
  Since $S$ is $\cV$-compatible we have $[x]_\cV\cap S=\emptyset$.
  Therefore, $[x]_\cV\cap\cl S\subset \cl S\setminus S=\mo S$.
  Clearly, due to Propositions \ref{prop:lcl} and \ref{prop:iso_inv_is_conv}, $\cl x\subset\mo S\subset \cl S$. Hence,
  \[
    y\in \Pi_\cV(x)\cap\cl S = ([x]_\cV\cup\cl x)\cap\cl S = ([x]_\cV\cap\cl S) \cup (\cl x \cap\cl S) \subset \mo S.
  \]
  To see \ref{it:ind_pair_2} note that by Proposition \ref{prop:iso-is-vcomp} the set $S$ is $\cV$-compatible and
    \[ \Piv(S)=\bigcup_{x\in S} \cl x\cup[x]_\cV = \bigcup_{x\in S} \cl x\cup S = \cl S.
    \]
  Thus, if $x\in S$, then $\Pi_\cV(x) \setminus \cl S = \emptyset$. 
  Therefore, $\Pi_\cV(x) \setminus \cl S \neq \emptyset$ for $x\in P_1=\cl S$ implies $x\in \cl S\setminus S=\mo S$.

  Finally, directly from the definition of mouth we have $\cl S \setminus \mo S=S$, which proves \ref{it:ind_pair_3}, as well as the fact that $(\cl S,\mo S)$ is saturated.
\end{proof}

We write $P\subset Q$ for index pairs $P$, $Q$ meaning $P_i\subset Q_i$ for $i=1,2$.
We say that index pairs $P$, $Q$ of $S$
are {\em semi-equal} if $P\subset Q$ and either $P_1=Q_1$ or
$P_2=Q_2$. For semi-equal pairs $P$, $Q$, we let
\[ A(P,Q):=\begin{cases}
             Q_1\setminus P_1 & \text{ if $P_2=Q_2$,}\\
             Q_2\setminus P_2 & \text{ if $P_1=Q_1$.}
          \end{cases}
\]

\begin{figure}
  \includegraphics[width=0.4\textwidth]{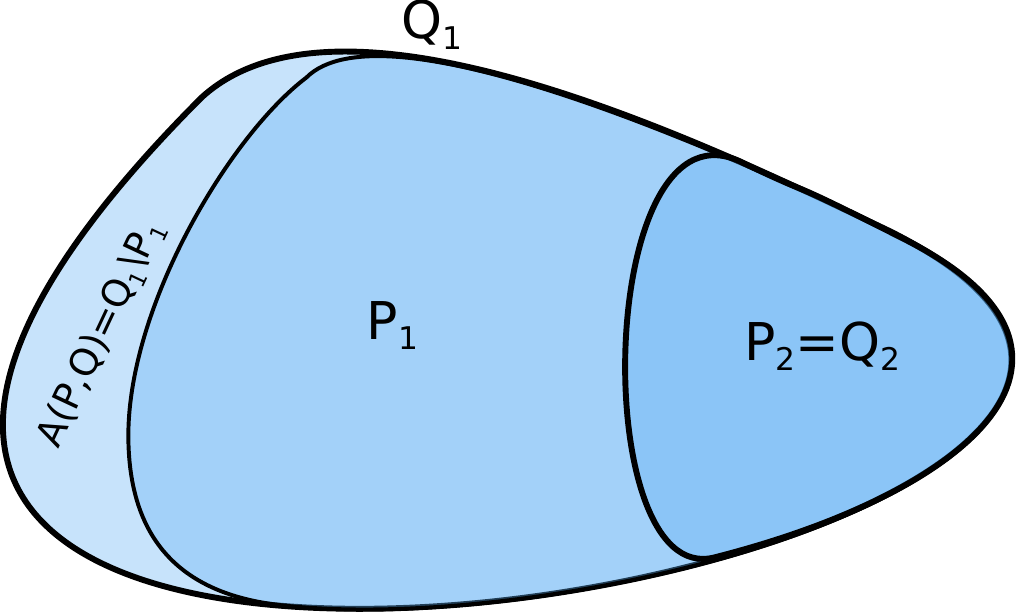}
  \hfill
  \includegraphics[width=0.4\textwidth]{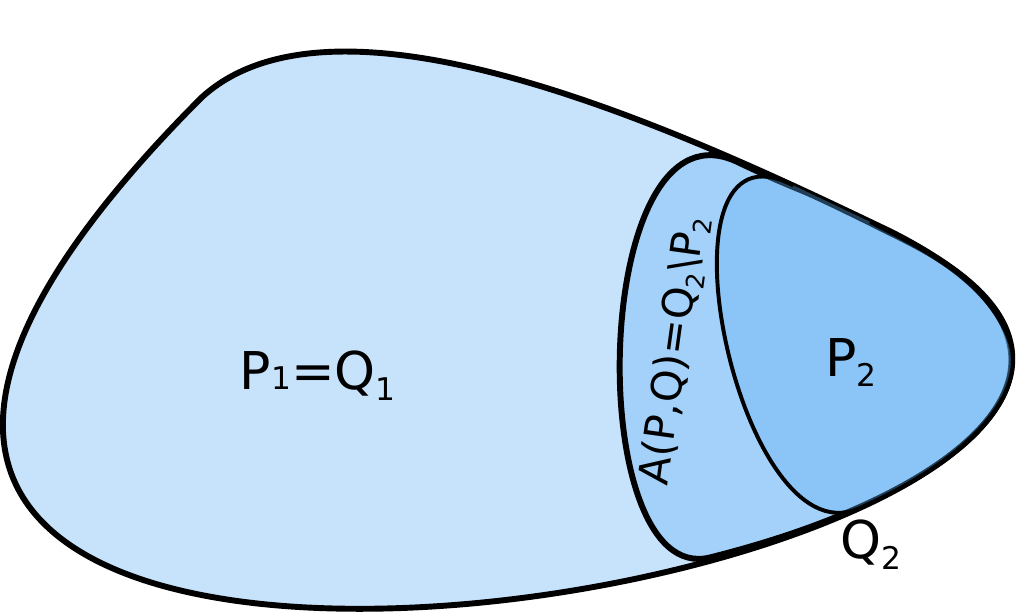}
  \caption{Schematic depiction of the two cases of a set $A(P,Q)$.}
\end{figure}

\begin{prop}\label{prop:no_full_sol_in_A}
  Let $P$ and $Q$ be semi-equal index pairs for $S$. 
  Then there is no essential solution in the set $A(P,Q)$.
\end{prop}
\begin{proof}
  First note that the definition of~$A(P,Q)$ implies either
  \[ A(P,Q) = Q_1\setminus P_1 \subset Q_1\setminus P_2 = Q_1\setminus Q_2
     \quad\mbox{and}\quad
     A(P,Q) \cap (P_1 \setminus P_2) = \emptyset , \]
  or
  \[ A(P,Q) = Q_2\setminus P_2 \subset Q_1\setminus P_2 = P_1\setminus P_2
     \quad\mbox{and}\quad
     A(P,Q) \cap (Q_1 \setminus Q_2) = \emptyset . \]
  Therefore, by \ref{it:ind_pair_3} and the first inclusions in the above two
  statements we get $\inv A(P,Q)\subset S$. Yet, the second identities above
  clearly show that $A(P,Q)$ is disjoint from~$S$.
  Thus, $\inv A(P,Q) = \emptyset$, and by the definition of the invariant part (see~(\ref{eq:inv_set})) there is no essential solution in $A(P,Q)$.
\end{proof}

\begin{lem}\label{lem:iso_of_saturated_idx_pairs}
  Assume $S$ is an isolated invariant set. 
  Let $P$ and $Q$ be saturated index pairs for $S$.
  Then $H(P_1,P_2)\cong H(Q_1,Q_2)$.
\end{lem}
\begin{proof}
  By the definition of a saturated index pair $Q_1\setminus Q_2 = S = P_1\setminus P_2$. 
  Hence, using Theorem \ref{thm:excision_ftop} we get $H(P_1,P_2)\cong H(Q_1,Q_2)$.
\end{proof}

\begin{prop}\label{prop:P1-P2_vcomp_and_conx}
  Assume $S$ is an isolated invariant set. 
  Let $P$ be an index pair for $S$. 
  Then the set $P_1 \setminus P_2$ is $\cV$-compatible and locally closed.
\end{prop}
\begin{proof}
  Assume that $P_1 \setminus P_2$ is not $\cV$-compatible. 
  This means that for some $x \in P_1\setminus P_2$ there exists a $y\in [x]_\cV \setminus (P_1 \setminus P_2)$. 
  Then $y\in P_2$ or $y\not\in P_1$.
  Consider the case $y\in P_2$. 
  Since $[x]_\cV=[y]_\cV$, we have $x\in\Pi_\cV(y)$. 
  It follows from \ref{it:ind_pair_1} that $x\in P_2$, a contradiction.
  Consider now the case $y\not\in P_1$. 
  Then from \ref{it:ind_pair_2} one obtains $x\in P_2$, which is again a contradiction.
  Together, these cases imply that $P_1\setminus P_2$ is $\cV$-compatible.

  Finally, the local closedness of $P_1\setminus P_2$ follows immediately from Proposition \ref{prop:lcl}(iii).
\end{proof}

\begin{prop}\label{prop:A(P,Q)_vcomp_and_conx}
  Assume $S$ is an isolated invariant set. 
  Let $P \subset Q$ be semi-equal index pairs for $S$. 
  Then $A(P,Q)$ is $\cV$-compatible and locally closed.
\end{prop}
\begin{proof}
  First note that our assumptions give $P_2,Q_2\subset P_1$ and $P_2,Q_2\subset Q_1$.
  If $P_2=Q_2$, then 
    \[A(P,Q) =Q_1\setminus P_1 = (Q_1\setminus P_2)\setminus (P_1\setminus P_2) = (Q_1\setminus Q_2)\setminus (P_1\setminus P_2).\]
  If $P_1=Q_1$, then 
    \begin{align*}
      A(P,Q)&= Q_2\setminus P_2 = Q_2\cap P_2^c = (P_1\cap P_2^c)\cap Q_2\\ 
       &=(P_1\cap P_2^c)\cap (Q_1\cap Q_2^c)^c = (P_1\setminus P_2)\setminus(Q_1\setminus Q_2),
    \end{align*}
  where the superscript $c$ denotes the set complement in~$X$.
  Thus, by Proposition \ref{prop:P1-P2_vcomp_and_conx}, in both cases, $A(P,Q)$ may be represented as a difference of $\cV$-compatible sets. 
  Therefore, it is also $\cV$-compatible.

  The local closedness of $A(P,Q)$ follows from Proposition \ref{prop:lcl}.
\end{proof}

\begin{lem}\label{lem:trivial_scc_partition_gives_trivial}
  Let $A$ be a $\cV$-compatible, locally closed subset of $X$ such that there is no essential solution in $A$. 
  Then $H(\cl A, \mo A)=0$.
\end{lem}

\begin{proof}
  Let $\cA:=\{V\in\cV\mid V\subset A\}$. 
  Since $A$ is $\cV$-compatible, we have $A=\bigcup\cA$. 
  Let $\lesssim_\cA$ denote the transitive closure of the relation $\preceq_\cA$ in $\cA$ given for $V,W\in\cA$ by
    \begin{align}\label{eq:lem-partial-order-condition}
      V\preceq_\cA W\ \Leftrightarrow\ V\cap\cl W\neq\emptyset.
    \end{align}  
  We claim that $\lesssim_\cA$ is a partial order in $\cA$. 
  Clearly, $\lesssim_\cA$ is reflective and transitive. 
  Hence, we only need to prove that $\lesssim_\cA$ is antisymmetric. 
  To verify this, suppose the contrary. 
  Then there exists a cycle $V_n\preceq_\cA V_{n-1}\preceq_\cA\dots\preceq_\cA V_0=V_n$ with $n>1$ and $V_i\neq V_j$ for $i\neq j$ and $i,j\in\{1,2,\dots n\}$. 
  Since $V_i\cap\cl V_{i-1}\neq\emptyset$ we can choose $v_i\in V_i\cap\cl V_{i-1}$ and $v_{i-1}'\in V_{i-1}$ such that $v_i\in\cl v_{i-1}'$.
  Then $v_i\in\Piv(v_{i-1}')$ and $v_{i-1}'\in\Piv(v_{i-1})$.
  Thus, we can construct an essential solution 
    \[\dotsc \cdot v_n'\cdot v_1\cdot v_1'\cdot v_2\cdot v_2'\cdot\dotsc \cdot v_{n-1}'\cdot v_n\cdot v_n'\cdot v_1\cdot\dotsc. \]
  This contradicts our assumption and proves that $\lesssim_\cA$ is a partial order.

  Moreover, since a constant solution in a critical multivector is essential, all multivectors in $\cA$ have to be regular. 
  Thus, 
  \begin{equation}\label{eq:lem-zero-hom}
    H(\cl V, \mo V) = 0 \quad\text{for every}\quad V\in\cA.
  \end{equation}
  Since $\lesssim_\cA$ is a partial order, we may assume that $\cA=\{V_i\}_{i=1}^m$ where the numbering of $V_i$ extends the partial order $\lesssim_\cA$ to a linear order $\leq_\cA$, that is,
    \[ V_1 \leq_\cA V_2 \leq_\cA \dots \leq_\cA V_m.\]
  We claim that
  \begin{equation}\label{eq:lem-ij}
    i<j\quad\Rightarrow\quad\cl V_i\setminus V_j=\cl V_i.
  \end{equation}
  Indeed, if this were not satisfied, then $V_j\cap\cl V_i\neq\emptyset$ which, by the definition~(\ref{eq:lem-partial-order-condition}) of $\preceq_\cA$ gives $V_j\preceq_\cA V_i$
  as well as $V_j\lesssim_\cA V_i$, and therefore $j\leq i$, a contradiction.
  For $k\in\{0,1,\dots m\}$ define set $W_k:=\bigcup_{j=1}^kV_j$.
  Then $W_0=\emptyset$ and $W_m=A$.
  Now fix a $k\in\{0,1,\dots m\}$.  
  Observe that by (\ref{eq:lem-ij}) we have
  \[
    \cl W_k\setminus A 
      = \bigcup_{j=1}^k\cl V_j\setminus\bigcup_{j=1}^m V_j 
      = \bigcup_{j=1}^k\cl V_j\setminus\bigcup_{j=1}^k V_j
      = \cl W_k\setminus W_k = \mo W_k.
  \]
  Therefore,
  \[
    \mo W_k = \cl W_k\setminus A\subset\cl A\setminus A=\mo A.
  \]
  It follows that $W_k\cup\mo A = \cl W_k\cup\mo A$.
  Hence, the set $Z_k:=W_k\cup\mo A$ is closed.
  For $k>0$ we have
  \[
    Z_k\setminus Z_{k-1} = W_k\setminus W_{k-1}\setminus \mo A =
    V_k\cap A = V_k = \cl V_k\setminus\mo V_k.
  \]
  Hence, we get from Theorem \ref{thm:excision_ftop} and (\ref{eq:lem-zero-hom})
  \[
    H(Z_k,Z_{k-1}) = H(\cl V_k,\mo V_k)=0.
  \]
  Now it follows from the exact sequence of the triple $(Z_{k-1}, Z_k, \cl A)$ that
  \[
    H(\cl A, Z_k)\cong H(\cl A,Z_{k-1}).
  \]
  Note that $Z_0=W_0\cup \mo A = \mo A$ and $Z_m=W_m\cup\mo A=A\cup\mo A=\cl A$.
  Therefore, we finally obtain
  \[
    H(\cl A,\mo A)=H(\cl A, Z_0)\cong H(\cl A,Z_m) = H(\cl A, \cl A) = 0,
  \]
    which completes the proof of the lemma.
\end{proof}

\begin{lem}\label{lem:semieq_q1p1}
  Let $P \subset Q$ be semi-equal index pairs of an isolated invariant set $S$. 
  If $P_1 = Q_1$, then $H(Q_2, P_2) = 0$, and analogously, if $P_2 = Q_2$, then $H(Q_1, P_1) = 0$.
\end{lem}
\begin{proof}
  By Theorem \ref{prop:A(P,Q)_vcomp_and_conx} the set $A(P,Q)$ is locally closed and $\cV$-compatible. 
  Hence, the conclusion follows from Proposition \ref{prop:no_full_sol_in_A} and Lemma \ref{lem:trivial_scc_partition_gives_trivial}.
\end{proof}

\begin{lem}\label{lem:semieq_ip_iso}
  Let $P\subset Q$ be semi-equal index pairs of an isolated invariant set $S$. 
  Then $H(P_1, P_2)\cong H(Q_1, Q_2)$.
\end{lem}
\begin{proof}
  Assume $P_2 = Q_2$. 
  We get from Lemma \ref{lem:semieq_q1p1} that $H(Q_1,P_1)=0$.
  Using Theorem \ref{thm:exact-hom-seq-triple} for the triple $P_2\subset P_1\subset Q_1$
  then implies
  \[ H(P_1,P_2) \cong H(Q_1, P_2) = H(Q_1, Q_2). \]
  Similarly, if $P_1 = Q_1$ we consider the triple $P_2\subset Q_2\subset Q_1$ and obtain  
  \[H(P_1, P_2)=H(Q_1, P_2)\cong H(Q_1, Q_2).\]
\end{proof}

In order to show that two arbitrary index pairs carry the same homological information, we need to construct auxiliary, intermediate index pairs. 
To this end, we define the \emph{push-forward} and the \emph{pull-back} of a set $A$ in $B$.
\begin{align}
  \pipl(A,B) &:= \{ x \in B \mid\ \exists_{\varphi\in\paths_\cV(B)}\ \lep{\varphi} \in A,\ \rep{\varphi} = x \},\label{eq:pi_push_forward}\\
  \pimn(A,B) &:= \{ x \in B \mid\ \exists_{\varphi\in\paths_\cV(B)}\ \lep{\varphi} = x,\ \rep{\varphi} \in A\}\label{eq:pi_pull_back}
\end{align}

\begin{prop}\label{prop:closedness_of_pi}
  Let $A\subset X$ then $\pi^+_\cV(A,X)$ ($\pi^-_\cV(A,X)$) is closed (open) and $\cV$-compatible.
\end{prop}
\begin{proof}
  Let $x\in \pi^+_\cV(A,X)$ be arbitrary. 
  Then there exists a point $a\in A$ and a $\varphi\in\pathsv(a,x,X)$. 
  For any $y\in[x]_\cV$ the concatenation $\varphi\cdot y$ is also a path. 
  Thus, $\pi^+_\cV(A,X)$ is $\cV$-compatible. 

  To show closedness, take an $x\in\pi^+_\cV(A,X)$ and $y\in\cl x$. 
  By (\ref{eq:pi_push_forward}) there exists an $a\in A$ and a $\varphi\in\pathsv(a, x, X)$. 
  Then the path $\varphi\cdot y$ is a path from $A$ to $y$, implying that $y\in \pi^+_\cV(A,X)$. 
  Since $X$ is finite one obtains
  \[ 
    \cl\pipl(A,X)=
    \bigcup_{x\in\Pi_\cV^+(A,X)}\cl x = \pipl(A,X),
  \]
  and therefore $\pipl(A,X)$ is closed.
  The proof for $\pimn(A,X)$ is symmetric.
\end{proof}

Let $P$ be an index pair for $S$. 
Define the set $\hat{P} \subset P_1$ of all points $x\in P_1$ for which there exists no path in $P_1$ which starts in $x$ and ends in $S$, that is, 
\begin{align}\label{eq:hatP}
  \hat{P} := \{ x\in P_1\ |\ \pi^+_\cV(x,P_1)\cap S=\emptyset \} .
\end{align}

\begin{prop}\label{prop:p2_sub_hatP}
  If $P$ is an index pair for an isolated invariant set $S$, then $S\cap\hat{P}=\emptyset$ and $P_2\subset\hat{P}$.
\end{prop}
\begin{proof}
  The first assertion is obvious. 
  In order to see the second take an $x\in P_2$ and suppose that $x\not\in\hat{P}$. 
  This means that there exists a path $\varphi$ in $P_1$ such that $\lep{\varphi}=x$ and $\rep{\varphi}\in S$. 
  The condition \ref{def:index_pair_cond_1} of Definition \ref{def:index_pair} implies $\im\varphi\in P_2$.
  Therefore, $\rep{\varphi}\in P_2$ and $P_2\cap S\neq\emptyset$ which contradicts $S\subset P_1\setminus P_2$.
\end{proof}

\begin{prop}\label{prop:mouth_sub_hatP}
  If $P$ is an index pair for an isolated invariant set $S$, then $\mo S\subset\hat{P}$. 
  Moreover, $\Pi_\cV(S)\subset S\cup\hat{P}$.
\end{prop}
\begin{proof}
  To prove that $\mo S\subset \hat{P}$ assume the contrary. 
  Then there exists an $x\in\mo S$, such that $\pi^+_\cV(x,P_1)\cap S\neq\emptyset$.
  It follows that there exists a path $\varphi$ in $P_1$ from $x$ to $S$.
  Since $x\in\mo S\subset \cl S$, we can take a $y\in S$ such that $x\in \cl y\subset \Pi_\cV(y)$.
  It follows that $\psi:=y\cdot\varphi$ is a path in $P_1$ through $x$ with endpoints in $S$. 
  Since, by Proposition \ref{prop:P1_isolates}, $P_1$ isolates $S$, we get $x\in S$, a contradiction.

  Finally, by $\cV$-compatibility of $S$ guaranteed by Proposition \ref{prop:iso-is-vcomp},
  we have the inclusion $\Pi_\cV(S)=\cl S \subset S\cup \mo S\cup\hat{P}= S\cup \hat{P}$,
  which proves the remaining assertion.
\end{proof}

\begin{prop}\label{prop:hatP_closed}
  Let $P$ be an index pair for an isolated invariant set $S$. 
  Then the sets $\hat{P}$ and $\hat{P}\cup S$ are closed.
\end{prop}
\begin{proof}
  Let $x\in \hat{P}$ and let $y\in \cl x$. 
  Then $y\in\Pi_\cV(x)$. 
  Moreover, $y\in P_1$, because $P_1$ is closed. 
  Clearly, if $\varphi\in\pathsv(y,P_1)$, then $x\cdot\varphi\in\paths_\cV(x,P_1)$.
  Therefore $\pi^+_\cV(y,P_1)\subset\pi^+_\cV(x,P_1)$. 
  Since, by (\ref{eq:hatP}), the latter set is disjoint from $S$, so is the former one. 
  Therefore, $y\in\hat{P}$. It follows that $\hat{P}$ is closed.

  Proposition \ref{prop:mouth_sub_hatP} implies that 
  $\cl(S\cup\hat{P})=\cl S\cup \hat{P}=S\cup\mo{S}\cup\hat{P}=S\cup\hat{P}$, 
  which proves the closedness of $S\cup\hat{P}$.
\end{proof}

\begin{lem}\label{lem:index_pairs_star_sstar}
  If $P$ is an index pair for an isolated invariant set $S$, then $P^*:=(S\cup\hat{P}, P_2)$ is an index pair for $S$ and $P^{**}:=(S\cup\hat{P}, \hat{P})$ is a saturated index pair for $S$.
\end{lem}

\begin{proof}
  First consider $P^*$.
  By Proposition \ref{prop:hatP_closed} set $P^*_1=S\cup\hat{P}$ is closed. 
  By Proposition \ref{prop:p2_sub_hatP} we have $P_2\subset\hat{P}\subset S\cup\hat{P}$.

  Let $x\in P_2^*=P_2$ and let $y\in\Pi_\cV(x)\cap P_1^*$.
  Then $y\in\Pi_\cV(x)\cap P_1$. 
  It follows from \ref{it:ind_pair_1} for $P$ that $y\in P_2$.
  Thus, \ref{it:ind_pair_1} is satisfied for $P^*$.

  Now, let $x\in P_1^*=S\cup\hat{P}$ and suppose that there is a $y\in\Pi_\cV(x)\setminus P_1^*\neq\emptyset$.
  We have $x\not\in S$, because otherwise $\Piv(x)\subset\mo S\cup S=\cl S\subset \cl(S\cup\hat{P})$ and then Proposition \ref{prop:hatP_closed} implies 
  $\Pi_\cV(x)\subset S\cup\hat{P}\subset P_1^*$ which contradicts $\Piv(x)\setminus P_1^* \neq \emptyset$.
  Hence, $x\in\hat{P}$. 
  We have $y\not\in P_1$ because otherwise $y\in\pipl(x,P_1)\subset \hat{P}\subset P_1^*$, a contradiction. 
  Thus $\Piv(x)\setminus P_1\neq\emptyset$. 
  Since $x\in P_1^*\subset P_1$, by \ref{it:ind_pair_2} for $P$ we get $x\in P_2=P_2^*$. 
  This proves \ref{it:ind_pair_2} for $P^*$.

  Clearly, $P_1^*\setminus P_2^* = P_1^*\setminus P_2\subset P_1\setminus P_2$,
  and therefore we have the inclusion $\inv \left( P^*_1\setminus P^*_2 \right)\subset\inv \left( P_1\setminus P_2 \right)=S$. 
  To verify the opposite inclusion, let $x\in S$ be arbitrary. 
  Since $S$ is an invariant set, there exists an essential solution $\varphi\in\esol(x, S)$.
  We have
    \[\im\varphi\subset S\subset (\hat{P}\cup S)\setminus P_2 =P^*_1\setminus P^*_2, \]
  because $P_2\cap S=\emptyset$.
  Consequently, $x\in\inv(P^*_1\setminus P^*_2)$ and
  $S=\inv(P^*_1\setminus P^*_2)$. 
  Hence, $P^*$ also satisfies \ref{it:ind_pair_3}, which completes the proof that $P^*$ is an index pair for $S$.

  Consider now the second pair $P^{**}$. 
  Let $x\in P_2^{**}=\hat{P}$ be arbitrary and choose $y\in\Pi_\cV(x)\cap P_1^{**}=\Pi_\cV(x)\cap (\hat{P}\cup S)$. 
  Since $x\in\hat{P}$ we get from (\ref{eq:hatP}) that $\Piv(x)\cap S=\emptyset$. 
  Thus, $y\in \Piv(x)\cap\hat{P}\subset\hat{P}=P_2^{**}$.
  This proves \ref{it:ind_pair_1} for the pair $P^{**}$.

  To see \ref{it:ind_pair_2} take an $x\in P_1^{**}=\hat{P}\cup S$ and assume $\Piv(x)\setminus P_1^{**}\neq\emptyset$. 
  We cannot have $x\in S$, because then $\Piv(x)\subset\Piv(S)$ and Proposition \ref{prop:mouth_sub_hatP} implies $\Piv(x)\subset S\cup\hat{P}=P_1^{**}$, a contradiction.
  Hence, $x\in\hat{P}=P_2^{**}$ which proves \ref{it:ind_pair_2} for $P^{**}$.

  Finally, we clearly have $S\cap\hat{P}=\emptyset$. 
  Therefore, $(S\cup\hat{P})\setminus \hat{P}= S$ and
    \[ 
      \inv(P_1^{**}\setminus P_2^{**}) = \inv((S\cup\hat{P})\setminus \hat{P}) = \inv{S}= S.
    \]
  This proves that $P^{**}$ satisfies \ref{it:ind_pair_3} and that it is saturated.
\end{proof}

\begin{thm}\label{thm:index_pairs_isomorphism}
  Let $P$ and $Q$ be two index pairs for an invariant set $S$. 
  Then $H(P_1, P_2)\cong H(Q_1, Q_2)$.
\end{thm}
\begin{proof}
  It follows from Lemma \ref{lem:index_pairs_star_sstar} that $P^{*}\subset P$ as well as $P^*\subset P^{**}$ are semi-equal index pairs.  
  Hence, we get from Lemma \ref{lem:semieq_ip_iso} that 
  \begin{align*}
    H(P_1,P_2)&\cong H(P_1^{*},P_2^{*})\cong H(P_1^{**},P_2^{**}).
  \end{align*}
  Similarly, one obtains
  \begin{align*}
    H(Q_1,Q_2)&\cong H(Q_1^{*},Q_2^{*})\cong H(Q_1^{**},Q_2^{**}).
  \end{align*}
  Since both pairs $P^{**}$ and $Q^{**}$ are saturated, it follows from Lemma \ref{lem:iso_of_saturated_idx_pairs} that
    $H(P_1^{**},P_2^{**}) \cong H(Q_1^{**},Q_2^{**}).$
  Therefore, $H(P_1,P_2) \cong H(Q_1,Q_2)$.
\end{proof}

\subsection{Conley index}

We define the \emph{homology Conley index} of an isolated invariant set $S$ as $H(P_1,P_2)$ where $(P_1, P_2)$ is an index pair for $S$. 
We denote the homology Conley index of $S$ by $\con(S)$.
Proposition \ref{prop:minimal_index_pair} and Theorem~\ref{thm:index_pairs_isomorphism} guarantee that the homology Conley index is well-defined.

Given a locally closed set $A\subset X$ we define its \emph{$i$th Betti number} $\beta_i(A)$ and \emph{Poincar\'e polynomial} $p_A(t)$, respectively, as the $i$th Betti number and the Poincar\'e polynomial of the pair $(\cl A,\mo A)$, that is, $\beta_i(A):=\beta_i(\cl A, \mo A)$ and 
$p_A(t):=p_{\cl A, \mo A}(t)$ (see (\ref{eq:poincare_polynomial})).

The theorem used in the following Proposition originally comes from \cite{rybakowski_zehnder_1985}, but we use its more general version that was stated in \cite{Mr2017}.

\begin{prop}\label{prop:poincare_equation}
  If $(P_1,P_2)$ is an index pair for an isolated invariant set~$S$, then
  \begin{equation}\label{eq:poincare_equation}
    p_S(t) + p_{P_2}(t) = p_{P_1}(t) + (1+t)q(t),
  \end{equation}
  where $q(t)$ is a polynomial with non-negative coefficients.  
  Moreover, if
  \[  
    H(P_1) = H(P_2) \oplus H(\cl S, \mo S) 
  \]
  then $q(t)=0$.
\end{prop}

\begin{proof}
An index pair $(P_1, P_2)$ induces a long exact sequence of homology modules
  \begin{equation}\label{eq:poincare_equation_long_sequence}
    \dotsc \rightarrow H_n(P_2) \rightarrow H_n(P_1) \rightarrow H_n(P_1, P_2) \rightarrow H_{n-1}(P_2) \rightarrow \dotsc.
  \end{equation}
By Proposition \ref{prop:minimal_index_pair} and Theorem \ref{thm:index_pairs_isomorphism} we have $H(P_1, P_2)\cong H(\cl S, \mo S)$. Thus, we can replace (\ref{eq:poincare_equation_long_sequence}) with
  \begin{equation*}
    \dotsc \rightarrow H_n(P_2) \rightarrow H_n(P_1) \rightarrow H_n(\cl S, \mo S) \rightarrow H_{n-1}(P_2) \rightarrow \dotsc.
  \end{equation*}
In view of \cite[Theorem 4.6]{Mr2017} we further get
\[ p_{S}(t) + p_{P_2}(t) = p_{P_1}(t) + (1+t)q(t). \]
for some polynomial $q$ with non-negative coefficients. 
The second assertion follows directly from the second part of \cite[Theorem~4.6]{Mr2017} (see also \cite{rybakowski_zehnder_1985}).
\end{proof}

We say that an isolated invariant set $S$ \emph{decomposes} into the isolated invariant sets $S'$ and $S''$ if $\cl S'\cap S''=\emptyset$, $S'\cap\cl S''=\emptyset$, as well as $S=S'\cup S''$.

\begin{prop}
  Assume an isolated invariant set $S$ decomposes into the isolated invariant sets $S'$ and $S''$. Then $\sol(S)=\sol(S')\cup\sol(S'')$.
\end{prop}
\begin{proof}
  The inclusion $\sol(S')\cup\sol(S'')\subset\sol(S)$ is trivial. 
  To see the opposite inclusion, let $\varphi\in\sol(S)$. 
  We have to prove that $\im\varphi\subset S'$ or $\im\varphi\subset S''$.
  If this were not the case, then without loss of generality we can assume that there exists a $j\in\ZZ$ such that both $\varphi(j)\in S'$ and $\varphi(j+1)\in S''$ are satisfied. 
  This immediately implies $\varphi(j+1)\in\cl\varphi(j)\cup[\varphi(j)]_\cV$.
  We have $\varphi(j+1)\not\in[\varphi(j)]_\cV$, because otherwise the $\cV$-compatibility of $S'$ (see Propostition \ref{prop:iso-is-vcomp}) implies $\varphi(j+1)\in S'$ and, in consequence, $S'\cap S''\neq\emptyset$, a contradiction.
  Hence, $\varphi(j+1)\in\cl \varphi(j)\subset\cl S'$ which yields $\cl S'\cap S''\neq\emptyset$, another contradiction, proving that $\varphi\in\sol(S')\cup\sol(S'')$. 
\end{proof}

\begin{thm}\label{thm:additivity_con_idx}
  Assume an isolated invariant set $S$ decomposes into the isolated invariant sets $S'$ and $S''$. 
  Then we have
  \[\con(S)=\con(S')\oplus\con(S'').\]
\end{thm}
\begin{proof}
  In view of Proposition~\ref{prop:minimal_index_pair}, the two pairs $P=(\cl S',\mo S')$ and $Q=(\cl S'', \mo S'')$ are saturated index pairs for $S'$ and $S''$, respectively. 
  Consider the following exact sequence given by Theorem \ref{thm:relMV-ftop}:
  \begin{align}\label{eq:additivity_con_idx_1}
    \begin{split}
      \dotsc \rightarrow& H_n(P_1\cap Q_1, P_2\cap Q_2) 
      \rightarrow H_n(P_1, P_2)\oplus H_n(Q_1, Q_2)\\
      \rightarrow& H_n(P_1\cup Q_1, P_2\cup Q_2)
      \rightarrow H_{n-1}(P_1\cap Q_1, P_2\cap Q_2)
      \rightarrow \dotsc.
    \end{split}
  \end{align}
  Note that $S'\cap Q_2\subset S'\cap\cl S''=\emptyset$ and similarly $S''\cap P_2=\emptyset$. 
  Since both~$P$ and~$Q$ are saturated and $S'\cap S''=\emptyset$ we get
  \begin{align*}
    P_1\cap Q_1 &= (S'\cup P_2)\cap (S''\cup Q_2)\\
    &=(S'\cap S'')\cup(S'\cap Q_2)\cup(P_2\cap S'')\cup(P_2\cap Q_2)
    =P_2\cap Q_2.
  \end{align*}
  Thus, $H(P_1\cap Q_1, P_2\cap Q_2)=0$, which together with the exact sequence (\ref{eq:additivity_con_idx_1}) implies
  \begin{equation}\label{eq:con-additivity-iso1}
    H_*(P_1\cup Q_1,P_2\cup Q_2)\cong H_*(P_1,P_2)\oplus H_*(Q_1,Q_2).
  \end{equation}  
  Notice further that $S'\cap\cl S''=\emptyset$ implies $S'\setminus Q_2=S'$.
  Similarly $S''\setminus P_2=S''$.
  Therefore, one obtains the identity
  \begin{align*}
    (P_1\setminus P_2\setminus Q_2)\cup(Q_1\setminus Q_2\setminus P_2) = (S'\setminus Q_2)\cup(S''\setminus P_2) = S'\cup S'' = S.
  \end{align*}
  Hence, by Theorem \ref{thm:excision_ftop},
    \begin{equation}\label{eq:con-additivity-iso2}
      H(\cl S, \mo S)\cong H(P_1\cup Q_1, P_2\cup Q_2).
    \end{equation}
  Finally, from (\ref{eq:con-additivity-iso1}) and (\ref{eq:con-additivity-iso2}) we get
    \begin{align*}
      \con(S) &= H(\cl S, \mo S)\cong H(P_1\cup Q_1, P_2\cup Q_2) \\
      &\cong H(P_1, P_2)\oplus H(Q_1, Q_2)= \con(S')\oplus \con(S''),
    \end{align*}
  which completes the proof of the theorem.
\end{proof}

\section{Attractors, repellers and limit sets}\label{sec:attr-rep-lim}

In the rest of the paper we assume that a combinatorial multivector field~$\cV$ on a finite topological space~$X$ is fixed, and
that the space~$X$ is an invariant set. 
We need the invariance assumption to guarantee the existence of an essential solution through every point in $X$.
This assumption is not very restrictive, because if $X$ is not invariant, then we can replace the space $X$ by its
invariant part $\Inv X$ and the multivector field $\cV$ by its restriction $\cV_{\Inv X}$ (see Propositions~\ref{prop:subfield}~and~\ref{prop:invX}).

\subsection{Attractors, repellers and minimal sets}

We say that an invariant set $A\subset X$ is an \emph{attractor} if $\Pi_\cV(A) = A$. 
In addition, an invariant set $R\subset X$ is a \emph{repeller} if $\Pi^{-1}_\cV(R) = R$.

The following proposition shows that we can also express the concepts of attractor and repeller in terms of push-forward and pull-back.
\begin{prop}\label{prop:attr-push-forward}
  Let $A$ be an invariant set. 
  Then $A$ is an~attractor (a~repeller) in $X$ if and only if $\pipl(A,X)=A$ ($\pimn(A,X)=A$).
\end{prop}
\begin{proof}
  Let $A$ be an attractor.
  The inclusion $S\subset\pipl(S,X)$ is true for an arbitrary set.
  Suppose that there exists a $y\in\pipl(A,X)\setminus A$.
  Then by (\ref{eq:pi_push_forward}) we can find an $x\in A$ and $\varphi\in\pathsv(x,y,X)$.
  This implies that there exists a $k\in\ZZ$ such that $\varphi(k)\in A$ and $\varphi(k+1)\not\in A$.
  But $\varphi(k+1)\in\Piv(\varphi(k))\subset\Piv(A)=A$, a contradiction.
  Therefore, $\pipl(A,X)=A$.

  Now assume that $\pipl(A,X)=A$. 
  Again, by (\ref{eq:pi_push_forward}), we get $A=\pipl(A,X)=\Piv(\pipl(A,X))=\Piv(A)$.

  The proof for a repeller is analogous.
\end{proof}

%

\begin{thm}\label{thm:attractor_closedness}
  The following conditions are equivalent:
  \begin{enumerate}
    \item $A$ is an attractor,
    \item $A$ is closed, $\cV$-compatible, and invariant,
    \item $A$ is a closed isolated invariant set.
  \end{enumerate}
\end{thm}
\begin{proof}
  Let $A$ be an attractor. 
  It follows immediately from Propositions~\ref{prop:attr-push-forward} and~\ref{prop:closedness_of_pi} that condition {\it (1)\/} implies condition {\it (2)\/}.
  Moreover, Proposition~\ref{prop:inv_conv_vcomp_is_iso} shows that {\it (2)\/} implies {\it (3)\/}. 
  Finally, suppose that {\it (3)\/} holds. 
  By Proposition \ref{prop:iso-is-vcomp} set $A$ is $\cV$-compatible. 
  It is also closed. 
  Therefore, we have
  \[
    \Pi_\cV(A) = \bigcup_{x\in A}\cl x\cup\vclass{x} =
    \bigcup_{x\in A}\cl x \cup \bigcup_{x\in A}\vclass{x} = \cl A\cup A= A,
  \]
  which proves that $A$ is an attractor.
\end{proof}

\begin{thm}\label{thm:repeller_openess}
  The following conditions are equivalent:
  \begin{enumerate}
    \item $R$ is a repeller,
    \item $R$ is open, $\cV$-compatible, and invariant,
    \item $R$ is an open isolated invariant set.
  \end{enumerate}
\end{thm}
\begin{proof}
  Assume $R$ is a repeller.
  It follows from Propositions~\ref{prop:attr-push-forward} and~\ref{prop:closedness_of_pi} that condition {\it (1)\/} implies condition {\it (2)\/},
  and Proposition~\ref{prop:inv_conv_vcomp_is_iso} shows that~{\it (2)\/} implies~{\it (3)\/}.
  Finally, assume that condition {\it (3)\/} holds. Then $R$ is $\cV$-compatible by Proposition \ref{prop:iso-is-vcomp}. 
  The openness of $R$ and Proposition \ref{prop:preimage-piv-opn} imply
  \[
    \Pi^{-1}_\cV(R) = \bigcup_{x\in R}\opn x\cup\vclass{x} =
    \bigcup_{x\in R}\opn x \cup \bigcup_{x\in R}\vclass{x} = R,
  \]
  which proves that $R$ is a repeller.
\end{proof}

Let $\varphi$ be a full solution in $X$. 
We define the \emph{ultimate backward} and \emph{forward image} of $\varphi$  respectively by
\begin{align*}
  \uimm{\varphi}&:=\bigcap_{t\in\ZZ^-}\varphi\left( (-\infty,t] \right),\\
  \uimp{\varphi}&:=\bigcap_{t\in\ZZ^+}\varphi\left( [t,+\infty) \right).
\end{align*}
Note that in a finite space a descending sequence of sets eventually must become constant.
Therefore, we get the following result.
\begin{prop}\label{prop:uim-non-empty}
  There exists a $k\in\NN$ such that $\uimm{\varphi} = \varphi( (-\infty,-k])$ and $\uimp{\varphi} = \varphi( [k,+\infty))$.
  In particular, the sets $\uimm{\varphi}$ and $\uimp{\varphi}$ are always non-empty.  
\end{prop}

\begin{prop}\label{prop:esol-in-uim}
  If $\varphi$ is a left-essential (a right-essential) solution, then we can find an essential solution $\psi$ such that $\im\psi\subset\uimm\varphi$ ($\im\psi\subset\uimp\varphi$).
\end{prop}
\begin{proof}
    We only consider the case of a right-essential solution~$\varphi$.
  By Proposition \ref{prop:uim-non-empty} there exists a $k\in\ZZ$ such that $\uimp\varphi=\varphi([k,+\infty))$.
  We consider two cases.
  If $\uimp\varphi$ passes through a critical multivector, then we can easily build a stationary essential solution.
  In the second case, we have at least two different multivectors $V,W\in\cV$ such that $V\cap\uimp\varphi\neq\emptyset\neq W\cap\uimp\varphi$.
  Then there exist $t,s,u\in\ZZ$ with $k<t<s<u$ and $\varphi(t)\in V$, $\varphi(s)\in W$,
  and $\varphi(u)\in V$. But then the concatenation $\dots\cdot\varphi([t,u])\cdot\varphi([t,u])\cdot\dots$ is clearly essential.
\end{proof}

\begin{defn}
  We say that an invariant set $A\subset X$ is {\em minimal} if the only attractor in $A$ is the entire set $A$. 
\end{defn}

\begin{prop}\label{prop:minimal-is-scc}
  Let $A\subset X$ be an invariant set.  
  Then $A$ is minimal if and only if $A$ is a strongly connected set in $G_\cV$.
\end{prop}
\begin{proof}
  Let $A$ be a minimal invariant set. Suppose it is not strongly connected.
  Then we can find points $x,y\in A$ such that $\pathsv(x,y,A)=\emptyset$.
  Define $A':=\inv\pipl(x,A)$.
  Clearly, $y\not\in A'$.
  We will show that $A'$ is a nonempty attractor in $A$.

  Let $z\in\pipl(x,A)$. 
  Since $A$ is invariant there exists $\varphi\in\esolp(z, A)$.
  By Proposition \ref{prop:esol-in-uim} we can construct an essential solution $\psi$ such that $\im\psi\subset\uimp\varphi\subset\pipl(x,A)$.
  Thus, $A'$ is nonempty.
  
  Now suppose that $\pipl(A',A)\neq A'$.
  Then there exists an $a\in\pipl(A',A)\setminus A'$.
  It follows from~(\ref{eq:pi_push_forward}) that for every $b\in A'$ we have $\pathsv(a,b,A)=\emptyset$,
  since otherwise we could construct an essential solution through~$a$ which lies
  in~$\pipl(x,A)$.
  Using exactly the same reasoning as above, one can further show that $\inv(\pipl(A',A)\setminus A')\neq\emptyset$.
  But since we clearly have the inclusion $\inv(\pipl(A',A)\setminus A')\subset \inv(\pipl(x,A))= A'$, this leads to a contradiction.
  Thus, Proposition \ref{prop:attr-push-forward} shows that the set $A'$ is indeed an attractor, 
  which is nonempty and a proper subset of~$A$. Since this contradicts the minimality of~$A$,
  we therefore conclude that~$A$ is strongly connected.

  Now assume conversely that $A$ is strongly connected.
  It is clear that for any point $x\in A$ we get $\pipl(x,A)=A$.  
  It follows by Proposition \ref{prop:attr-push-forward} that the only attractor in $A$ is the entire set $A$.
\end{proof}

The duality allows to adapt the proof of Proposition \ref{prop:minimal-is-scc} to get the following proposition.
\begin{prop}
  An invariant set~$R$ is a minimal invariant set if and only if the only repeller in~$R$
  is the entire set~$R$.
\end{prop}

%

%

\begin{prop}\label{prop:minimal-intersect-attr}
  Let $S\subset X$ be a minimal invariant set and let $A\subset X$ be an attractor (a repeller).
  If $A\cap S\neq\emptyset$ then $S\subset A$.
\end{prop}
\begin{proof}
  Let $x\in A\cap S$ and let $y\in S$. 
  There exists $\varphi\in\pathsv(x,y,S)$.
  Now define $t=\min\dom{\varphi}$ and $s=\max\dom{\varphi}$.
  Clearly 
  \[
    \varphi(t+1)\in\Piv(\varphi(t))=\Piv(x)\subset\Piv(A)=A.
  \]
  Now, by induction let $k\in\{t,t+1,\dots,s-1\}$ and $\varphi(k)\in A$ then
  \[
    \varphi(k+1)\in\Piv(\varphi(k))\subset\Piv(A)=A.
  \]
  Therefore, $y=\varphi(k+1)\in A$, and this implies $S\subset A$.
\end{proof}

For a full solution $\varphi$ in $X$, define the sets 
\begin{align}\label{eq:valpha}
  \cVm(\varphi) := \left\{V\in\cV\mid V\cap\uimm\varphi\neq\emptyset\right\},\\
  \cVp(\varphi) := \left\{V\in\cV\mid V\cap\uimp\varphi\neq\emptyset\right\}.
\end{align}
%
We refer to a multivector $V\in\cVm(\varphi)$ (respectively $\cVp(\varphi)$) as a \emph{backward} (respectively \emph{forward}) \emph{ultimate multivector} of $\varphi$.
The families $\cVm(\varphi)$ and $\cVp(\varphi)$ will be used in the sequel, in particular in the proof of the following theorem.

\begin{thm}\label{thm:dual-attractor}
  Assume the whole space $X$ is invariant. Let $A\subset X$ be an attractor. 
  Then $A^\star:=\Inv\left(X\setminus A\right)$ is a repeller in $X$, which 
  is called the \emph{dual repeller} of $A$. 
  Conversely, if~$R$ is a repeller, then $R^\star:=\Inv\left(X\setminus R\right)$ is an attractor in $X$, 
  called the \emph{dual attractor} of~$R$.
  Moreover, the dual repeller (or the dual attractor) is nonempty, unless we have $A=X$ (or $R=X$).
\end{thm}
\begin{proof}
  We will show that $A^\star$ is open. 
  Let $x\in A^\star$ and let $y\in\opn x$. 
  Then one has $x\in\cl y$ by Proposition~\ref{prop:cl-in-ftop}.
  Since $A$ is closed as an attractor (Proposition~\ref{thm:attractor_closedness}), we immediately get $y\not\in A$. 
  The invariance of $X$ lets us select a $\varphi\in\esol(y,X)$. 
  Then $\im\varphi^-\cap A=\emptyset$, because otherwise there exists a $t\in\ZZ^-$ such that $\varphi(t)\in A$ and $\varphi(t+1)\not\in A$, which gives
  \[ 
    \varphi(t+1)\in\Piv(\varphi(t))\subset\Piv(A)=A,
  \]
  a contradiction.
  Now, let $\psi\in\esol(x,A^\star)$. 
  Clearly, $x\in\cl y\subset\Piv(y)$. 
  Thus, $\varphi^-\cdot\psi^+\in\esol(y,X\setminus A)$. 
  It follows that $y\in\inv(X\setminus A)=A^\star$ which proves that $\opn A^\star\subset A^\star$.
  Therefore, the set $A^\star$ is open.

  Since $A$ is $\cV$-compatible, also $X\setminus A$ is $\cV$-compatible. 
  Let $x\in A^\star$ and let $y\in[x]_\cV$. 
  Since $x\in A^\star\subset X\setminus A$, $\cV$-compatibility of $X\setminus A$ implies $y\not\in A$.
  Select a  $\varphi\in\esol(x,A^\star)$
  Thus, $\varphi^-\cdot y\cdot\varphi^+$ is a well-defined essential solution in $X\setminus A$, that is, $\esol(y,X\setminus A)\neq\emptyset$. 
  It follows that $y\in\inv(X\setminus A)=A^\star$. 
  Hence, $A^\star$ is $\cV$-compatible.
  Altogether, the set $A^\star$ is invariant, open and $\cV$-compatible. 
  Thus, by Theorem \ref{thm:repeller_openess} it is a repeller. 

  
  Finally, we will show that $A^\star\neq\emptyset$ unless $A=X$.
  Suppose that $X\setminus A\neq\emptyset$, and let $x\in X\setminus A$.
  Since~$X$ is invariant, there exists a $\varphi\in\esol(x,X)$.
  As in the first part of the proof one can show that $\im\varphi^-\cap A=\emptyset$,
  that is, we have $\im\varphi^- \subset X \setminus A$. According to
  Proposition~\ref{prop:esol-in-uim} there exists an essential solution $\psi$
  such that $\im\psi \subset \uimm\varphi \subset \im\varphi^- \subset X \setminus A$,
  and this immediately implies~$A^\star = \Inv\left(X\setminus A\right) \neq\emptyset$.
\end{proof}

\subsection{Limit sets}

We define the \emph{$\cV$-hull of a set $A\subset X$} as the intersection of all $\cV$-compatible, locally closed sets containing $A$, and denote it by $\langle A\rangle_\cV$.
As an immediate consequence of Proposition \ref{prop:lcl-intersection} and Proposition \ref{prop:union-intersection-vcomp} we get the following result.
\begin{prop}\label{prop:vhull-vcomp-lcl}
  For every $A\subset X$ its $\cV$-hull is $\cV$-compatible and locally closed.
\end{prop}
We define the \emph{$\alpha$- and $\omega$-limit sets of a full solution $\varphi$} respectively by
\begin{align*}
  \alpha(\varphi) &:= \left\langle\uimm\varphi \right\rangle_{\cV},\\
  \omega(\varphi) &:= \left\langle\uimp\varphi \right\rangle_{\cV}.
\end{align*}

The following proposition is an immediate consequence of Proposition \ref{prop:lcl-op}.
\begin{prop}\label{prop:dual-full-solution}
  Assume $\varphi$ is a full solution of $\cV$ and $\varphi^{\op}$ is the associated dual solution of $\cV^{\op}$. 
  Then
  \begin{align*}
    \alpha(\varphi) = \omega(\varphi^{\op}) \quad\text{and}\quad \omega(\varphi) = \alpha(\varphi^{\op}).
  \end{align*}
\end{prop}

\begin{prop}\label{prop:valpha-alpha}
  Let $\varphi$ be an essential solution. 
  Then
  \begin{align*}
    \alpha(\varphi) = \left\langle\bigcup\valpha(\varphi) \right\rangle_{\hspace{-1mm}\cV}
  \end{align*}
  and
  \begin{align*}
    \omega(\varphi) = \left\langle\bigcup\vomega(\varphi) \right\rangle_{\hspace{-1mm}\cV}.
  \end{align*}
\end{prop}
\begin{proof}
  Clearly 
  \[  
    \uimm\varphi \subset \bigcup\left\{V\in\cV\mid V\cap\uimm\varphi\neq\emptyset\right\} = \bigcup\valpha(\varphi)  
  \]
  and therefore 
  \[  
    \alpha(\varphi) = \left\langle \uimm\varphi \right\rangle_\cV
      \subset \left\langle\bigcup\valpha(\varphi) \right\rangle_{\hspace{-1mm}\cV}.  
  \]
  Now let $x\in \bigcup\valpha(\varphi)$. 
  Then there exists a $y\in\vclass{x}$ such that $y\in \uimm\varphi$. 
  Then $y\in\alpha(\varphi)$ and, since $\alpha(\varphi)$ is $\cV$-compatible, $\vclass{y}=\vclass{x}\subset \alpha(\varphi)$. 
  Thus, we have $\bigcup{\valpha({\varphi})}\subset \alpha(\varphi)$. 
  Since $\alpha(\varphi)$ is locally closed and $\cV$-compatible, the set $\alpha(\varphi)$ is a superset of the $\cV$-hull of $\bigcup\alpha_\cV(\varphi)$.
Hence, 
  \[ 
    \left\langle\bigcup\valpha(\varphi) \right\rangle_{\hspace{-1mm}\cV} \subset \alpha(\varphi).
  \]
The proof for $\omega(\varphi)$ is analogous.
\end{proof}

\begin{lem}\label{lem:v-valpha-in-out}
  Assume $\varphi:\ZZ\rightarrow X$ is a full solution of $\cV$ and $\valpha(\varphi)$ (respectively $\vomega(\varphi)$) contains at least two different multivectors.
  Then for every $V\in\cV$ such that $V\subset\alpha(\varphi)$ (respectively $V\subset\omega(\varphi)$) we have
  \begin{align}\label{eq:lem-v-valpha-in-out-1}
    \left(\Piv(V)\setminus V \right)\cap\alpha(\varphi)\neq\emptyset
    \;\;\; \text{(respectively }
    \left(\Piv(V)\setminus V \right)\cap\omega(\varphi)\neq\emptyset
    \text{)}
  \end{align}
  and 
  \begin{align}\label{eq:lem-v-valpha-in-out-2}
    \left(\Piv^{-1}(V)\setminus V \right)\cap\alpha(\varphi)\neq\emptyset
    \;\;\; \text{(respectively }
    \left(\Piv^{-1}(V)\setminus V \right)\cap\omega(\varphi)\neq\emptyset
    \text{).}
  \end{align}
\end{lem}
\begin{proof}
  Assume $V\in\cV$ is such that $V\subset\alpha(\varphi)$.
  This happens if $V\in\valpha(\varphi)$, but might also happen for some $V\not\in\valpha(\varphi)$.
  
  Assume first that $V\in\valpha(\varphi)$.
  Since there are at least two different multivectors in the set $\valpha(\varphi)$ there exists a strictly decreasing sequence $k:\NN\rightarrow\ZZ^-$ such that $\varphi(k_n)\in V$ and $\varphi(k_n+1)\not\in V$.
  Since the set $\{\varphi(k_n+1)\mid n\in\NN\}\subset X$ is finite, after taking a subsequence, if necessary, we may assume that $\varphi(k_n+1)=y\not\in V$.
  Let $W:=[y]_\cV$.
  Then $W\neq V$ and $y\in W\cap\uimm\varphi\cap\Piv(V)$.
  This implies $W\in\valpha(\varphi)$ and $\Piv(V)\cap W\neq\emptyset$.
  
  By Proposition \ref{prop:valpha-alpha} we have 
  \[ 
    \emptyset\neq \Piv(V)\cap W\subset (\Piv(V)\setminus V) \cap W
    \subset \left(\Piv(V)\setminus V\right) \cap \alpha(\varphi).
  \]
  Thus, (\ref{eq:lem-v-valpha-in-out-1}) is satisfied.

  Now assume that $V\not\in\valpha(\varphi)$. 
  We have $\Piv(V)=\cl V \cup V=\cl V$. 
  Suppose that (\ref{eq:lem-v-valpha-in-out-1}) does not hold. 
  Then
  \[
    \emptyset=\left(\Piv(V)\setminus V\right)\cap\alpha(\varphi)
    =\left(\cl V\setminus V\right)\cap\alpha(\varphi)
    =\mo V\cap\alpha(\varphi)
  \]
  and therefore
  \begin{align*}
    \alpha(\varphi)\setminus V 
    &= \left(\cl\alpha(\varphi)\setminus\mo\alpha(\varphi)\right)\setminus\left(V\cup\mo V\right) \\
    &= \left(\cl\alpha(\varphi)\setminus\mo\alpha(\varphi)\right)\setminus\cl V
    = \cl\alpha(\varphi)\setminus\left(\mo\alpha(\varphi)\cup\cl V\right).
  \end{align*}
  By Proposition \ref{prop:lcl-in-ftop} the set $\alpha(\varphi)\setminus V$ is locally closed as a difference of closed sets. 
  Clearly, $\alpha(\varphi)\setminus V$ is $\cV$-compatible. 
  This shows that $\alpha(\varphi)$ is not a minimal locally closed and $\cV$-compatible set containing $\bigcup\valpha(\varphi)$. 
  This contradicts Proposition \ref{prop:valpha-alpha}. 
  Hence (\ref{eq:lem-v-valpha-in-out-1}) holds for $V\subset\alpha(\varphi)$.
  
  The proof of (\ref{eq:lem-v-valpha-in-out-1}) for $V\in\vomega(\varphi)$ is a straightforward adaptation of the proof for $V\subset\alpha(\varphi)$.
  To see (\ref{eq:lem-v-valpha-in-out-2}) observe that since $\varphi^{\op{}}$ is a full solution
      of~$\cV^{\op{}}$, $\omega(\varphi^{\op{}})=\alpha(\varphi)$ 
    by Proposition \ref{prop:dual-full-solution} and, clearly, 
    $\vomega(\varphi^{\op{}})=\valpha(\varphi)$, 
    we may apply (\ref{eq:lem-v-valpha-in-out-1}) to 
    $\cV^{\op{}}$, $\varphi^{\op{}}$ and $\omega(\varphi^{\op{}})$.
  Thus, by Proposition \ref{prop:piv-dual} we get
  \[
    \left(\Piv^{-1}(V)\setminus V\right)\cap \alpha(\varphi) =
    \left(\Pi_{\cV^{\op{}}}(V)\setminus V\right)\cap \omega(\varphi^{\op{}})\neq\emptyset,
  \]
  and the claim for $\omega(\varphi)$ follows similarly.
\end{proof}

\begin{thm}\label{thm:limit-set-is-iso-inv}
Let $\varphi$ be an essential solution in $X$. 
Then both limit sets $\alpha(\varphi)$ and $\omega(\varphi)$ are non-empty strongly connected isolated invariant sets.
\end{thm}
\begin{proof}
  The nonemptiness of $\alpha(\varphi)$ and $\omega(\varphi)$ follows from Proposition \ref{prop:uim-non-empty}. 
 
  The sets $\alpha(\varphi)$ and $\omega(\varphi)$ are $\cV$-compatible and locally closed by Proposition \ref{prop:vhull-vcomp-lcl}. 
  In order to prove that they are isolated invariant sets it suffices to apply Proposition \ref{prop:inv_conv_vcomp_is_iso} as long as we prove that $\alpha(\varphi)$ and $\omega(\varphi)$ are also invariant.

  We will first prove that $\alpha(\varphi)$ is invariant. 
  Let $x\in\alpha(\varphi)$.  
  Suppose that $\valpha(\varphi)$ is a singleton. 
  Then by Proposition \ref{prop:valpha-alpha}, $\alpha(\varphi)=\vclass{x}$. 
  Since $\varphi$ is essential, this is possible only if $\vclass{x}$ is critical. 
  It follows that the stationary solution $\psi(t)=x$ is essential. 
  Hence $\alpha(\varphi)$ is an isolated invariant set.

  Assume now that there are at least two different multivectors in $\valpha(\varphi)$. 
  Then the assumptions of Lemma \ref{lem:v-valpha-in-out} are satisfied and, as a consequence of (\ref{eq:lem-v-valpha-in-out-1}), for every $x\in\alpha(\varphi)$ there exist a point $x'\in\vclass{x}$ and a $y\in\alpha(\varphi)$ such that $y\in\left(\Piv(x')\setminus \vclass{x} \right)\cap\alpha(\varphi)$. 
  Hence, we can construct a right-essential solution 
  \[ 
    x_0\cdot x'_0\cdot x_1\cdot x'_1\cdot x_2\cdot x'_2\cdot\dotsc,  
  \]
  where $x_0=x$, $x'_i \in \vclass{x_i}$, and $x_{i+1}\in\left(\Piv(x'_i)\setminus \vclass{x_i} \right)\cap\alpha(\varphi)$. 
  Property (\ref{eq:lem-v-valpha-in-out-2}) provides a complementary left-essential solution. 
  Concatenation of both solutions gives an essential solution in $\alpha(\varphi)$. 
  Hence, we proved that $\alpha(\varphi)$ is invariant and consequently an isolated invariant set.

  Finally, we prove that $\alpha(\varphi)$ is strongly connected.
  To this end consider points $x,y\in\alpha(\varphi)$.
  We will show that then $\pathsv(x,y,\alpha(\varphi))\neq\emptyset$.
  Using the two abbreviations $V_x:=\vclass{x}$ and $V_y:=\vclass{y}$
  it is clear that
  \begin{align}\label{eq:thm-limit-iso-inv-1}
    V_x,V_y\in\valpha(\varphi) 
    \quad\Rightarrow\quad
    \pathsv(x,y,\alpha(\varphi))\neq\emptyset.
  \end{align}

  Assume now that $V_x\subset\alpha(\varphi)\setminus\bigcup\valpha(\varphi)$ and $V_y\in\valpha(\varphi)$.
  By (\ref{eq:thm-limit-iso-inv-1}) it is enough to show that there exists at least one point $z\in\bigcup\valpha(\varphi)$ such that $\pathsv(x, z,\alpha(\varphi))\neq\emptyset$.
  Suppose the contrary.
  The set $\pipl(V_x,\alpha(\varphi))$ is closed and $\cV$-compatible in $\alpha(\varphi)$ by Proposition \ref{prop:closedness_of_pi}.
  By Proposition \ref{prop:lcl-cl-sub} set $A:=\alpha(\varphi)\setminus\pipl(V_x,\alpha(\varphi))$ is locally closed.
  Clearly, it is $\cV$-compatible and contains $\bigcup\valpha(\varphi)$.
  Yet this results in a contradiction, because we have now found a smaller $\cV$-hull for $\bigcup\valpha(\varphi)$.

  Now, consider the case when $V_x\in\valpha(\varphi)$ and $V_y\subset\alpha(\varphi)\setminus\bigcup\valpha(\varphi)$.
  The set $\pipl(V_x,\alpha(\varphi))$ is $\cV$-compatible and locally closed by Proposition~\ref{prop:closedness_of_pi}.
  In view of (\ref{eq:thm-limit-iso-inv-1}) we have $\bigcup\valpha(\varphi)\subset\pipl(V_x,\alpha(\varphi))$.  
  Thus, one either has $V_y\subset\pipl(V_x,\alpha(\varphi))$ or $\pipl(V_x,\alpha(\varphi))$ is a smaller $\cV$-compatible, locally closed set containing $\bigcup\valpha(\varphi)$. In both cases we get a contradiction.

  Finally, let $V_x,V_y\subset\alpha(\varphi)\setminus\bigcup\valpha(\varphi)$ and let $z\in\bigcup\valpha(\varphi)$.
  Using the previous cases we can find $\psi_1\in\pathsv(x,z,\alpha(\varphi))$ and $\psi_2\in\pathsv(z,y,\alpha(\varphi))$.
  Then, $\psi_1\cdot\psi_2\in\pathsv(x,y,\alpha(\varphi))$. 
  This finishes the proof that $\alpha(\varphi)$ is strongly connected.

  The proof for $\omega(\varphi)$ is analogous.
\end{proof}

Let $\varphi$ be an essential solution. 
We say that an isolated invariant set $S$ \emph{absorbs $\varphi$ in positive} (respectively negative) \emph{time} if $\varphi(t)\in S$ for all $t\geq t_0$ (for all $t\leq t_0$) for some $t_0\in\ZZ$. 
We denote by $\Omega(\varphi)$ (respectively by $\cA(\varphi)$) the family of isolated invariant sets absorbing $\varphi$ in positive (respectively negative) time.
\begin{prop}\label{prop:limits-cap-absorb}
  For an essential solution $\varphi$ we have
  \begin{align}
    \alpha(\varphi)&=\bigcap \cA(\varphi),\label{eq:prop-limits-cap-absorb-1}\\
    \omega(\varphi)&=\bigcap\Omega(\varphi).\label{eq:prop-limits-cap-absorb-2}
  \end{align} 
\end{prop}
\begin{proof}
  Let $\varphi$ be an essential solution.
  It follows from Proposition \ref{prop:uim-non-empty} that there exists a $k\in\ZZ^-$ such that $\varphi((-\infty,k])=\uimm\varphi\subset\alpha(\varphi)$.
  Moreover, by Proposition \ref{thm:limit-set-is-iso-inv} we have $\alpha(\varphi)\in \cA(\varphi)$.
  Hence $\bigcap\cA(\varphi)\subset\alpha(\varphi)$.
  To see the opposite inclusion take an $S\in\cA(\varphi)$.
  Then, there exists a $t_0\in\ZZ^-$ such that $\varphi((-\infty,t_0])\subset S$.
  It follows that 
  \[ 
    \alpha(\varphi)=\langle\uimm\varphi\rangle_\cV\subset \langle\varphi((-\infty, t_0])\rangle_\cV \subset \langle S\rangle_\cV=S.
  \]
  Hence, $\alpha(\varphi)\subset\cA(\varphi)$. This proves (\ref{eq:prop-limits-cap-absorb-1}).
  The proof of (\ref{eq:prop-limits-cap-absorb-2}) is analogous.
%
\end{proof}

Let $A,B\subset X$. 
We define the \emph{connection set from $A$ to $B$} by:
\begin{equation}\label{eq:connection_set}
  C(A,B) := \left\{x\in X\ |\ \exists_{\varphi\in\esol(x,X)}\ \alpha(\varphi)\subset A\ \text{and}\ \omega(\varphi)\subset B \right\}.
\end{equation}

\begin{prop}\label{prop:invariance_of_connection}
  Assume $A,B\subset X$. 
  Then the connection set $C(A,B)$ is an isolated invariant set.
\end{prop}
\begin{proof}
  To prove that $C(A,B)$ is invariant, take an $x\in C(A,B)$ and choose a $\varphi\in\esol(x,X)$ as in (\ref{eq:connection_set}). 
  It is clear that $\varphi(t)\in C(A,B)$ for every $t\in\ZZ$. 
  Thus, $\varphi\in\esol(x,C(A,B))$, and this in turn implies
  $x\in\inv C(A,B)$ and shows that~$C(A,B)$ is invariant.
  Now consider a point $y\in\vclass{x}$. 
  Then the solution $\rho=\varphi^-\cdot y\cdot\varphi^+$ is a well-defined essential solution through $y$ such that $\alpha(\rho)\subset A$ and $\omega(\rho)\subset B$. 
  Thus, $C(A,B)$ is $\cV$-compatible. 

  In order to prove that $C(A,B)$ is locally closed, consider $x,z\in C(A,B)$, and a $y\in X$ such that $z\leq_\cT y\leq_\cT x$. 
  Select essential solutions $\varphi_x\in\esol(x,C(A,B))$ and $\varphi_z\in\esol(z,C(A,B))$. 
  Then $\psi:=\varphi^-_x\cdot y\cdot\varphi^+_z$ is a well-defined essential solution through $y$ such that $\alpha(\psi)\subset A$ and $\omega(\psi)\subset B$.
  It follows that $y\in C(A,B)$. 
  Thus, by Proposition \ref{prop:lcl-in-ftop}, $C(A,B)$ is locally closed. 
  Finally, Proposition \ref{prop:inv_conv_vcomp_is_iso} proves that $C(A,B)$ is an isolated invariant set.
\end{proof}

\begin{prop}\label{prop:empty_conn_attr_dual_rep}
  Assume $A$ is an attractor. 
  Then $C(A, A^\star)=\emptyset$. 
  Similarly, if $R$ is a repeller, then $C(R^\star,R)=\emptyset$.
\end{prop}
\begin{proof}
  Suppose there exists an $x\in C(A, A^\star)$. 
  Then by (\ref{eq:connection_set}) we can choose a $\varphi\in\esol(x,X)$ and a $t\in\ZZ$ such that $\varphi(t)\in A$ and $\varphi(t+1)\not\in A$. 
  However, since $A$ is an attractor, $\varphi(t)\in A$ implies $\varphi(t+1)\in A$, a contradiction.
  The proof for a repeller is analogous.
\end{proof}

\section{Morse decomposition, Morse equation, Morse inequalities}\label{sec:morse-decomposition}

In this final section we define Morse decompositions and prove the Morse inequalities for combinatorial multivector fields.
We recall the general assumption that $\cV$ is a fixed combinatorial vector field on a finite topological space $X$
and $X$ is invariant. 

\subsection{Morse decompositions}

\begin{defn}\label{def:morse_decomposition}
  Assume $X$ is invariant and  $(\PP,\leq)$ is a finite poset.
  Then the collection $\cM=\setof{M_p\mid p\in \PP}$ is called a {\em Morse decomposition of $X$} if
  the following conditions are satisfied:
  \begin{enumerate}[label=(\roman*)]
    \item\label{it:md_i} $\cM$ is a family of mutually disjoint, isolated invariant subsets of $X$.
    \item\label{it:md_ii} For every essential solution $\varphi$ in $X$ either $\im{\varphi}\subset M_r$ for an $r\in\PP$ or
       there exist $p,q\in\PP$ such that $q > p$ and
    \begin{equation*}
       \alpha(\varphi)\subset M_q\quad \text{and}\quad \omega(\varphi)\subset M_p.
    \end{equation*}
  \end{enumerate}
  We refer to the elements of $\cM$ as \emph{Morse sets}.
\end{defn}

Note that in the classical definition of Morse decomposition the analogue of condition {\it \ref{it:md_ii}\/} is formulated in terms of trajectories passing through points $x \not\in \bigcup\cM$. 
In our setting we have to consider all possible solutions. 
There are two reasons for that: the non-uniqueness of a solution passing through a point and the tightness of finite topological spaces. 
In particular, in the finite topological space setting it is possible to have a non-trivial Morse decomposition such that every point is contained in a Morse set. 
Without our modification of the definition of Morse decomposition, recurrent behavior spreading into several sets is a distinct possibility. 
Figure \ref{fig:morse-dec-count} illustrates such an example.

\begin{figure}[tb]
  \includegraphics[width=0.4\textwidth]{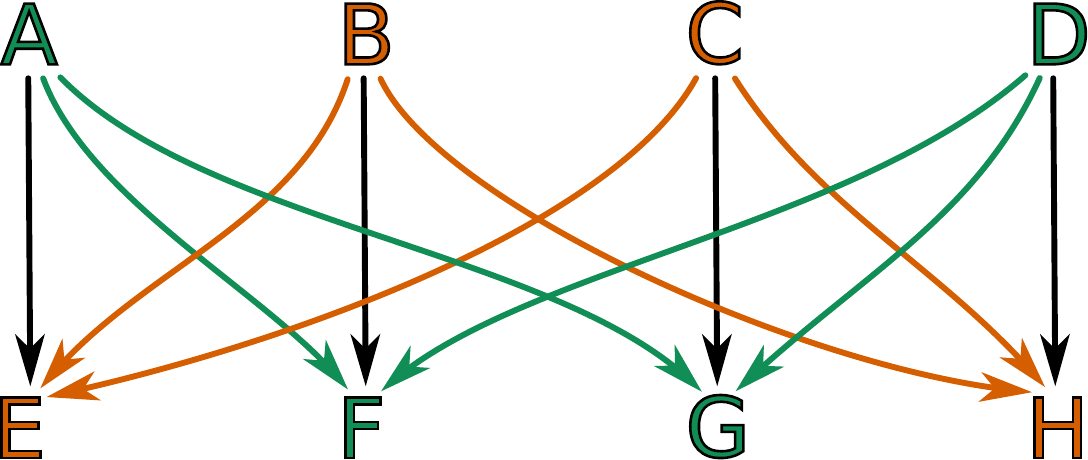}
  \caption{
    A sample combinatorial multivector field $\cV=\left\{\{A,D,F,G\}, \{B,C,E,H\}\right\}$ on the finite topological space $X=\left\{ A, B,C,D,E,F,G,H\right\}$ with Alexandroff topology induced by the partial order indicated by arrows.
    If we consider $\cM=\cV$, then one obtains a partition into isolated invariant sets
    with $X\setminus\cM=\emptyset$. 
    Note that $\dots D\cdot H\cdot B\cdot F\cdot D\cdot\dots$ is a periodic trajectory which passes through both ``Morse sets.''
  }
  \label{fig:morse-dec-count}
\end{figure}


\begin{prop}\label{prop:attr-rep-pair-ms}
    Let~$X$ be an invariant set, 
  let $A\subset X$ be an attractor, and let $A^\star$ denote its nonempty dual repeller. 
  Furthermore, define $M_1:=A$, $M_2:=A^\star$, and let $\PP:=\{1,2\}$ be an indexing set with the order induced from~$\NN$.
  Then $\cM=\{M_1, M_2\}$ is a Morse decomposition of $X$.
\end{prop}
\begin{proof}
  By Theorems \ref{thm:attractor_closedness} and \ref{thm:repeller_openess} both $A$ and $A^\star$ are isolated invariant sets which are clearly disjoint.
  Let $x\in X$ and let $\varphi\in\esol_\cV(x,X)$.
  By Theorem~\ref{thm:limit-set-is-iso-inv} the set $\omega(\varphi)$ is strongly connected and invariant. 
  It is also minimal by Proposition \ref{prop:minimal-is-scc}.
  By Proposition $\ref{prop:minimal-intersect-attr}$ it is either a subset of $A$ or a subset of $\inv (X\setminus A)=A^\star$.
  The same holds for $\alpha(\varphi)$.

  We therefore have four cases.
  The situation $\alpha(\varphi)\subset M_2$ and $\omega(\varphi)\subset M_1$ is consistent with the definition.
  The case $\alpha(\varphi)\subset M_1$ and $\omega(\varphi)\subset M_2$ is clearly in conflict with the definition of an attractor and a repeller.
  Now suppose that we have $\alpha(\varphi)\subset M_1$ and $\omega(\varphi)\subset M_1$. 
  It follows that there exists a $t\in\ZZ$ such that $\varphi((-\infty,t]))\subset A$. 
  Since $A$ is an attractor we therefore have $\varphi(t+1)\in \Piv(\varphi(t)) \subset A$,
  and induction easily implies $\im\varphi\subset A=M_1$.
  The same argument holds for $M_2$.
\end{proof}

\subsection{Strongly connected components as Morse decomposition}

We recall that $G_\cV$ stands for the digraph interpretation of the multivalued map~$\Piv$ associated with the multivector field $\cV$ on $X$. 
\begin{thm}\label{thm:sccs_morse_decomp}
  Assume $X$ is invariant. Consider the family $\cM$ of all strongly connected components~$M$ of~$G_\cV$ with~$\esol(M)\neq\emptyset$.
  Then $\cM$ is a minimal Morse decomposition of $X$.
\end{thm}
\begin{proof}
  For convenience, assume that $\cM=\{M_i\mid i\in\PP\}$ is bijectively indexed by a finite set $\PP$.
  Any two strongly connected components $M_i,M_j\in\cM$ are clearly disjoint and by Theorem \ref{thm:scc_is_iso_inv} they are isolated invariant sets. 
  Hence, condition {\it \ref{it:md_i}} of a Morse decomposition is satisfied.

  Define a relation $\leq$ on the indexing set $\PP$ by
  \[ 
    i\leq j\ \Leftrightarrow\ 
    \exists_{\varphi\in\pathsv(X)}\ 
    \lep{\varphi}\in M_j\ \text{and}\ \rep{\varphi}\in M_i.
  \]
  It is clear that $\leq$ is reflexive.
  To see that it is transitive consider $M_i,M_j,M_k\in\cM$ such that $k\leq j\leq i$. 
  It follows that there exist paths $\varphi$ and $\psi$ such that $\lep{\varphi}\in M_i$, $\rep{\varphi}, \lep{\psi}\in M_j$ and $\rep{\psi}\in M_k$.
  Since $M_j$ is strongly connected we can find $\rho\in\pathsv(\rep{\varphi},\lep{\psi},X)$. 
  The path $\varphi\cdot\rho\cdot\psi$ clearly connects $M_i$ with $M_k$ proving that $k\leq i$.

  In order to show that $\leq$ is antisymmetric consider sets $M_i,M_j$ with $i\leq j$ and $j\leq i$.
  It follows that there exist paths $\varphi$ and $\psi$ such that $\lep{\varphi},\rep{\psi}\in M_i$ and $\rep{\varphi},\lep{\psi}\in M_j$. 
  Since the sets $M_i,M_j$ are strongly connected we can find paths $\rho$ and $\rho'$ from $\rep{\varphi}$ to $\lep{\psi}$ and from $\rep{\psi}$ to $\lep{\varphi}$ respectively. 
  Clearly, $\varphi\in\pathsv(\lep{\varphi},\rep{\varphi}, X)$ and $\rho\cdot\psi\cdot\rho'\in(\rep{\varphi},\lep{\varphi}, X)$. 
  This proves that $M_i$ and $M_j$ are the same strongly connected component.

  Let $x\in X$ and let $\varphi\in\esol(x,X)$. 
  We will prove that $\alpha(\varphi)\subset M_q$ and $\omega(\varphi)\subset M_p$ for some $M_p,M_q\in\cM$. 
  Note that $\varphi^{-1}(V)$ is right-infinite for any $V\in\vomega(\varphi)$. 
  It follows that for any $V, W\in \vomega(\varphi)$ we can find $0<t_0<t_1<t_2$
  such that $\varphi(t_0),\varphi(t_2)\in V$ and $\varphi(t_1)\in W$.
  Thus, points in $\bigcup\vomega(\varphi)$ are in the same strongly connected component and therefore 
    $\bigcup\omega_\cV(\varphi)\subset C$ for some strongly path connected component of $G_\cV$. 
  Moreover, $\dotsc\cdot\varphi\restr{[t_0,t_2]}\cdot\varphi\restr{[t_0,t_2]}\cdot\dotsc$ is clearly an essential solution in $C$. 
  Thus $C=M_p\in\cM$ for some $p\in\PP$.
  By Proposition \ref{prop:scc-Vcomp-lcl} the set $M_p$ is $\cV$-compatible and locally closed. 
  Hence, $M_p$ is a superset of the $\cV$-hull of $\bigcup\vomega(\varphi)$ and from Proposition \ref{prop:valpha-alpha} we get $\omega(\varphi)\subset M_p$. 
  A similar argument gives $\alpha(\varphi)\subset M_q$ for some $q\in\PP$. 
  It is clear from the definition of $\leq$ that $p\leq q$.

  Next, we show that $\alpha(\varphi)\subset M\in\cM$ and $\omega(\varphi)\subset M$ implies $\im\varphi\subset M$. 
  Thus, take a $y\in\im\varphi$.
  Then $y=\varphi(t_1)$ for some $t_1\in\ZZ$. 
  Since $\alpha(\varphi)$ and $\omega(\varphi)$ are subsets of $M$, we can find $t_0<t_1$ and $t_2>t_1$, such that $x:=\varphi(t_0)\in M$ and $z:=\varphi(t_2)\in M$.
  Since $M$ is strongly connected there exists a path $\rho$ from $z$ to $x$.
  Then $\varphi\restr{[t_1,t_2]} \cdot \rho\in\pathsv(y,x,X)$.
  Since $\varphi\restr{[t_0,t_1]}\in\pathsv(x,y,X)$, we conclude that $y$ belongs to the strongly connected component of $x$, that is, $y\in M$.
  This completes the proof that $\cM$ is a Morse decomposition.

  To show that $\cM$ is a minimal Morse decomposition assume the contrary.
  Then we can find a Morse decomposition $\cM'$ of an $M\in\cM$ with at least two different Morse sets $M_1$ and $M_2$ in $\cM'$. 
  Since $M_1$ and $M_2$ are disjoint and $\cV$-compatible we can find disjoint multivectors $V_1\subset M_1$ and $V_2\subset M_2$.
  Since the set $M$ is strongly connected we can find paths $\varphi\in\pathsv(x,y,M)$ and $\rho\in\pathsv(y,x,M)$ with $x\in V_1$ and $y\in V_2$. 
  The alternating concatenation of these paths $\psi:=\dotsc\cdot\varphi\cdot\rho\cdot\varphi\cdot\rho\dotsc$ is a well-defined essential solution. 
  Then $\emptyset\neq\im\psi\subset\alpha(\psi)\cap\omega(\psi)$ which implies $\im\psi\subset M_3$ for an $M_3\in\cM'$.
  However, $\im\psi\cap M_1\neq\emptyset\neq\im\psi\cap M_2$, a contradiction.
\end{proof}

\subsection{Morse sets}

For a subset $I\subset\PP$ we define the \emph{Morse set of~$I$} by
\[ M(I):=\bigcup_{i,j\in I} C(M_i,M_j).\]

\begin{thm}\label{thm:MI_iso_inv}
  The set $M(I)$ is an isolated invariant set.
\end{thm}
\begin{proof}
  Observe that $M(I)$ is invariant, because, by Proposition \ref{prop:invariance_of_connection}, every connection set is invariant, and by Proposition \ref{prop:union-of-invariant-sets} the union of invariant sets is invariant.
  We will prove that $M(I)$ is locally closed. 
  To see that, suppose the contrary. 
  Then, by Proposition \ref{prop:lcl-in-ftop}, we can choose $a,c\in M(I)$ and a point $b\not\in M(I)$ such that $c\leq_\cT b\leq_\cT a$.
  There exist essential solutions $\varphi_a\in\esol(a,X)$ and $\varphi_c\in\esol(c,X)$ such that $\alpha(\varphi_a)\subset M_q$ and $\omega(\varphi_c)\subset M_p$ for some $p,q\in I$. 
  It follows that $\psi:=\varphi_a^-\cdot b\cdot\varphi_c^+$ is a well-defined
  essential solution such that $\alpha(\psi)\subset M_q$ and $\omega(\psi)\subset M_p$. 
  Hence, $b\in C(M_q,M_p)\subset M(I)$ which proves that $M(I)$ is locally closed.
  Moreover, $M(I)$ is $\cV$-compatible as a union of $\cV$-compatible sets.
  Thus, the conclusion follows from Proposition \ref{prop:inv_conv_vcomp_is_iso}.
\end{proof}

\begin{thm}\label{thm:down-I-attractor}
  If $I$ is a down set in $\PP$, then $M(I)$ is an attractor in $X$.
\end{thm}
\begin{proof}
  We will show that $M(I)$ is closed. 
  For this, let $x\in\cl(M(I))$. 
  By Proposition \ref{prop:cl-as-union} we can choose a $y\in M(I)$ such that $x\in\cl y$. 
  Consider essential solutions $\varphi_x\in\esol(x,X)$ and $\varphi_y\in\esol(y,M(I))$ with $\alpha(\varphi_y)\subset M_i$ for some $i\in I$. 
  The concatenated solution $\varphi:=\varphi_y^-\cdot\varphi_x^+$ is well-defined and satisfies $\alpha(\varphi)\subset M_i$ and $\omega(\varphi)\subset M_j$ for some $j\in\PP$.
  Definition \ref{def:morse_decomposition} implies that $i>j$.
  Since $I$ is a down set, we get $j\in I$. 
  It follows that $x\in C(M_i,M_j)\subset M(I)$.
  Thus, $\cl M(I)\subset M(I)$, which proves that $M(I)$ is closed.
  Finally, Theorem \ref{thm:attractor_closedness} implies that the set $M(I)$ is an attractor.
\end{proof}

\begin{thm}\label{thm:idx_pair_for_convex_i}
  If $I\subset\PP$ is convex, then $(M(I^\leq), M(I^<))$ is an index pair for
  the isolated invariant set $M(I)$.
\end{thm}
\begin{proof}
  By Proposition \ref{prop:A-leq-down-set} the sets $I^\leq$ and $I^<$ are down sets. 
  Thus, by Theorem \ref{thm:down-I-attractor} both $M(I^\leq)$ and $M(I^<)$ are attractors. 
  It follows that $\Pi_\cV(M(I^\leq))\subset M(I^\leq)$ and $\Pi_\cV(M(I^<))\subset M(I^<)$. 
  Therefore, conditions \ref{it:ind_pair_1} and \ref{it:ind_pair_2} of an index pair are satisfied.

  Let $A:= M(I^\leq)\setminus M(I^<)$. 
  The set $A$ is $\cV$-compatible as a difference of $\cV$-compatible sets. 
  By Proposition \ref{prop:lcl} it is also locally closed, because $M(I^\leq)$ and~$M(I^<)$
  are closed as attractors (see Theorems \ref{thm:attractor_closedness} and \ref{thm:repeller_openess}).
  We claim that $M(I)\subset A$. 
  To see this, assume the contrary and select an $x\in M(I)\setminus A$.
  By the definition of $M(I)$ we can find an essential solution $\varphi$ through $x$ such that $\omega(\varphi)\subset M_p$ for some $p\in I$.
  Since $M(I)\subset M(I^\leq)$ and $x\not\in A$ we get $x\in M(I^<)$.
  But $M(I^<)$ is an attractor.
  Therefore $\omega(\varphi)\subset M(I^<)$, which in turn implies $p\not\in I$, a contradiction.

  To prove the opposite inclusion take an $x\in\inv(M(I^\leq)\setminus M(I^<))$.
  Then we can find an essential solution $\varphi\in\esol(x,M(I^\leq)\setminus M(I^<))$,
  and clearly one has $\im\varphi\subset M(I^\leq)\setminus M(I^<)$.
  In particular, 
  \begin{align}\label{eq:thm-idx-pair-for-convex-i-1}
    \uimm\varphi\cap M(I^<)=\emptyset
    \quad\text{and}\quad
    \uimp\varphi\cap M(I^<)=\emptyset .
  \end{align}
  We also have $\varphi\in\esol(x, M(I^\leq))$, which means that there exist $p,q\in I^\leq$ such that $p\geq q$, $\alpha(\varphi)\subset M_p$, $\omega(\varphi)\subset M_q$.
  We cannot have $p\in I^<$, because then we get $\emptyset\neq\uimm\varphi\subset\alpha(\varphi)\subset M_p\subset M(I^<)$ which contradicts (\ref{eq:thm-idx-pair-for-convex-i-1}).
  Therefore, $p\in I^\leq\setminus I^<=I$.
  By an analogous argument we get $q\in I$.
  It follows that $x\in C(M_p,M_q)\subset M(I)$.
\end{proof}

Since for a down set $I\subset\PP$ we have $I^\leq=I$, $I^<=\emptyset$, as an immediate consequence of Theorem \ref{thm:idx_pair_for_convex_i} we get the following corollary.

\begin{cor}\label{cor:idx_pair_for_lower_set}
  If $I$ is a down set in $\PP$, then $I^\leq = I$, $I^<=\emptyset$, $(M(I),\emptyset)$ is an index pair for $M(I)$.
\end{cor}

\begin{thm}
  Assume $X$ is invariant, $A\subset X$ is an attractor and $A^\star$ is its dual repeller. 
  Then we have
  \begin{equation}\label{eq:poincare_eq_for_attr_rep_pair}
  p_A(t) + p_{A^\star}(t) = p_X(t) + (1+t)q(t)
  \end{equation}
  for a polynomial $q(t)$ with non-negative coefficients. 
  Moreover, if $q\neq 0$, then $C(A^\star, A)\neq\emptyset$.
\end{thm}

\begin{proof}
  Let $\PP:=\{1,2\}$ with order induced from $\NN$, $M_1:=A$ and $M_2:=A^\star$. 
  Then $\cM:=\{M_1,M_2\}$ is a Morse decomposition of $X$ by Proposition~\ref{prop:attr-rep-pair-ms}. 
  For $I:=\{2\}$ one obtains $I^\leq=\{1,2\}$ and $I^<=\{1\}$. 
  Yet, this immediately implies both $M(I^\leq)=X$ and $M(I^<)=M(\{1\})=A$.
  We have
  \begin{equation}\label{eq:poincare_eq_for_attr_rep_pair_prop}
  p_X(t)=p_{M(I^\leq)}(t) \quad\mbox{and}\quad p_A(t)=p_{M(I^<)}(t).
  \end{equation}
  By Theorem \ref{thm:idx_pair_for_convex_i} the pair $(M(I^\leq), M(I^<))$ is an index pair for $M(I)=A^\star$. 
  Thus, by substituting $P_1:=M(I^\leq)$, $P_2:=M(I^<)$, $S:=A^\star$ into~(\ref{eq:poincare_equation}) in Corollary \ref{prop:poincare_equation} we get (\ref{eq:poincare_eq_for_attr_rep_pair}) from (\ref{eq:poincare_eq_for_attr_rep_pair_prop}). 
  By Proposition \ref{prop:empty_conn_attr_dual_rep} we have the identity $C(A,A^\star)=\emptyset$. 
  Therefore, if in addition $C(A^\star, A)=\emptyset$, then $X$ decomposes into $A$ and $A^\star$, and Theorem~\ref{thm:additivity_con_idx} implies
  \begin{equation*}
    H(P_1)=\con(X)=\con(A)\oplus \con(A^\star)=H(P_2)\oplus H(A^\star),
  \end{equation*}
  as well as $q=0$ in view of Proposition \ref{prop:poincare_equation}. 
  This finally shows that $q\neq 0$ implies $C(A^\star, A)\neq\emptyset$.
\end{proof}

\subsection{Morse equation and Morse inequalities}
The following two theorems follow from the results of the preceding section by adapting the proofs of the corresponding results in \cite{Mr2017}.

\begin{thm}\label{thm:morse-equation}
  Assume $X$ is invariant and $\PP=\{1,2,...,n\}$ is ordered by the linear order of the natural numbers. 
  Let $\cM:=\{M_p\ |\ p\in\PP\}$ be a Morse decomposition of  $X$ and set $A_i:= M(\{i\}^{\leq})$, $A_0:=\emptyset$. 
  Then $(A_{i-1},M_i)$ is an attractor-repeller pair in $A_i$. 
  Moreover,
    \begin{equation*}
        \sum_{i=1}^{n} p_{M_i}(t) = p_X(t) + (1+t)\sum_{i=1}^n q_i(t)
    \end{equation*}
    for some polynomials $q_i(t)$ with non-negative coefficients and such that $q_i(t)\neq 0$ implies $C(M_i,A_{i-1})\neq\emptyset$ for $i=2,3,...,n$.
\end{thm}

As before, for a locally closed set $A\subset X$ we define its $k$th Betti number by
$\beta_k(A):=\rank H_k(\cl A, \mo A)$.
\begin{thm}\label{thm:morse-inequalities}
    Assume $X$ is invariant. For a Morse decomposition $\cM$ of~$X$ define
    \[
        m_k(\cM):=\sum_{r\in\PP}\beta_k(M_r).
    \]
    Then for any $k\in\ZZ^+$ we have the following inequalities.
    \begin{enumerate}[(i)]
        \item The strong Morse inequalities:
            \[ m_k(\cM) - m_{k-1}(\cM) + ... \pm m_0(\cM) \geq \beta_k(X) - \beta_{k-1}(X) + ... \pm \beta_0(X), \]
        \item The weak Morse inequalities:
            \[ m_k(\cM) \geq \beta_k(X).\]
    \end{enumerate}
\end{thm}


\bibliographystyle{abbrv}
\bibliography{references-ds,references-top}

\end{document}